\documentclass{amsart}
\usepackage{amsfonts}
\usepackage{bigints}
\usepackage{graphicx}
\usepackage{graphics,color}
\usepackage{amsmath}
\usepackage{amssymb}
\usepackage{morefloats}
\usepackage{caption} 
\newtheorem{theorem}{Theorem}
\theoremstyle{plain}

\newtheorem{corollary}{Corollary}

\newtheorem{lemma}{Lemma}

\newtheorem{remark}{Remark}

\numberwithin{equation}{section}
\input{tcilatex}

\begin{document}
\title[A Second Order Ensemble Method]{A second-order Ensemble method based on a blended backward differentiation formula timestepping scheme for time-dependent Navier-Stokes equations}
\author{Nan Jiang}
\address[All authors]{Department of Scientific Computing, Florida State University, Tallahassee,
FL 32306, USA}
\email{njiang@fsu.edu}
\urladdr{https://people.sc.fsu.edu/\symbol{126}njiang}
\thanks{The research of the author described herein was supported by the US Air Force Office of Scientific Research grant FA9550-15-1-0001 and the US Department of Energy Office of Science grant DE-SC0010678.}
\thanks{}
\date{\today}
\keywords{Navier-Stokes Equations, Ensemble Calculation, Uncertainty Quantification, Blended BDF}
\dedicatory{ }
\begin{abstract}
We present a second-order ensemble method based on a blended three-step backward differentiation formula (BDF) timestepping scheme to compute an ensemble of Navier-Stokes equations. Compared with the only existing second-order ensemble method that combines the two-step BDF timestepping scheme and a special explicit second-order Adams-Bashforth treatment of the advection term, this method is more accurate with nominal increase in computational cost. We give comprehensive stability and error analysis for the method. Numerical examples are also provided to verify theoretical results and demonstrate the improved accuracy of the method.
\end{abstract}

\maketitle

\section{Introduction}
{\allowdisplaybreaks
Uncertainty quantification in geophysical systems as well as many engineering processes often involves computing an ensemble of nonlinear partial differential equations (PDE), see for instance \cite{CCDL05}, \cite{Lewis05}, \cite{LP08}, \cite{MX06}, \cite{TK93}. Solving these nonlinear PDEs numerically is usually very demanding in both computer resources and computing time, as even one single realization may require millions or even billions of degrees of freedom to obtain useful approximations. If the nonlinear effect is dominant, accurate approximations are hard to obtain, especially if the computational domain is large, e.g., global forecasting systems for numerical weather predictions. Computing ensembles inevitably results in a huge increase in the computational cost and poses a great challenge in performing accurate ensemble calculation. In the past few decades, most efforts have been devoted to developing data assimilation methods to reduce the number of realizations required, \cite{TK93}, \cite{BP95}. Only recently, an ensemble algorithm has been developed by Jiang and Layton \cite{JL14}, \cite{JL15} to compute an ensemble of time dependent Navier-Stokes equations efficiently. Instead of treating the simulation of each realization as separate tasks, this novel algorithm solves all realizations at one pass for each time level. It takes advantage of the efficiency of computing a linear system with multiple right hands for which highly efficient algorithms have been established and well studied, i.e., Block CG \cite{FOP95}, Block QMR \cite{FM97}, Block GMRES \cite{GS96}. As a result, this new ensemble algorithm can reduce the computational cost significantly compared to the usual routine of computing the realizations separately.

Stability and accuracy are critical aspects in the development of such algorithms. In \cite{JL14}, an ensemble time stepping scheme based on a combination of backward Euler and forward Euler is studied. Using the finite element method for spacial discretization, the method is proven to be long time stable and first order convergent under a CFL-like time step condition. This condition depends on Reynolds number and degrades quickly as Reynolds number grows. To relax it, two ensemble eddy viscosity numerical regularization methods are proposed in \cite{JL15}. They stabilize the system by adding extra numerical dissipation parameterized by mixing length and kinetic energy in fluctuations. A time relaxation regularization is also studied in \cite{TNW15}. It is also reported in \cite{TNW15} that grad-div stabilization can significantly weaken the time step restriction. As higher order methods are more efficient and thus more desirable in real engineering problems, developing accurate higher order ensemble methods is of great scientific and engineering interest. Nevertheless, extending the usual higher order timestepping schemes to the ensemble algorithm is not trivial, as the ensemble methods require different time discretizations for different terms to ensure its efficiency as well as stability. The only existing higher order ensemble method \cite{J15}, which we will refer to as (En-BDF2), is based on a two-step Backward Differentiation Formula (BDF2) and a special explicit second order in time Adams-Bashforth (AB2) treatment of the advection term. In this paper, we study a new second order ensemble method that is more accurate than (En-BDF2).

Classical BDF time schemes are among the most popular methods in the field of computational fluid dynamics (CFD) due to their strong stability properties, \cite{BDK82}, \cite{E04}, \cite{H01}. The highest order strongly A-stable BDF method is well-known to be the two-step BDF method. Classical BDF2 has been extensively used for large scale scientific computations as it can be used with large time steps without encountering numerical instability. Higher order multi-step BDF schemes are more accurate and efficient but less stable (not A stable). Hence one obvious approach is to blend the classical BDF2 and a classical higher order BDF method to obtain a multi-step method that is more accurate than classical BDF2 but still preserves good stability properties, such as A-stability. A family of such methods is proposed in \cite{NSMH03}, which blends the classical BDF2 and BDF3 method with a tuning parameter $\gamma$. The schemes are given by
\begin{gather}
D_{\gamma}(u_t^{n+1})= \gamma \left[\frac{3u^{n+1}-4u^n+u^{n-1}}{2\Delta t}\right]
+(1-\gamma)\left[\frac{\frac{11}{6}u^{n+1}-3u^n+\frac{3}{2}u^{n-1}-\frac{1}{3}u^{n-2}}{\Delta t}\right],\nonumber
\end{gather}
where $\gamma\in [\frac{1}{2}, 1]$. These are three-step methods with smaller coefficient on the leading term of the truncation error than classical BDF2, \cite{NJ11}. When $\gamma=\frac{1}{2}$, the scheme has the smallest truncation error constant, which is exactly half of the classical BDF2 scheme, \cite{VCL10}. This time marching scheme has been extensively tested in modern CFD codes in the area of aerodynamics, such as FUN3D developed and maintained at NASA Langley. In this paper, we propose a new second order ensemble method to efficiently compute an ensemble of Navier-Stokes equations based on this blended BDF scheme.

We consider an ensemble of $J$ Navier-Stokes equations with different initial conditions and/or different body forces, $j=1, ..., J$: 
\begin{align}
u_{j,t}+u_{j}\cdot\nabla u_{j}-\nu\triangle u_{j}+\nabla p_{j} &
=f_{j}(x,t)\text{, in }\Omega\text{, } \label{eq:NSE}\\
\nabla\cdot u_{j} &  =0\text{, in }\Omega\text{,}\nonumber\\
u_{j} &  =0\text{, on }\partial\Omega\text{,}\nonumber\\
u_{j}(x,0) &  =u_{j}^{0}(x)\text{, in }\Omega\text{,}\nonumber
\end{align}
where $\Omega$ is an open, regular domain in $\mathbb{R}^{d}$ $(d=2\text{ or }
3)$.

To construct a stable efficient ensemble algorithm, we need to use different time discretizations for different terms. The essential idea of the efficient ensemble algorithm is that all ensemble members share the same coefficient matrix and the main difficulty arises from the nonlinear term. Thus we split the nonlinear term into two terms with one containing the mean velocity that is independent of the index of ensemble members and the other one containing the fluctuation that characterizes each realization. The nonlinear term with the fluctuation needs to be lagged to previous time levels so it will go to the right hand side of the linear systems to be solved, so that the coefficient matrix is independent of the index of ensemble members. One consequence of lagging this term is a CFL-like condition to ensure the stability of the ensemble method. Now the key is to define an ensemble mean that is compatible with the blended BDF scheme. Let $t^{n}:=n\Delta t, n=0,1,2,...,N_T,$ and $T:= N_T\Delta t$ and denote $u_{j}^{n}=u_{j}(t^{n})$, $j=1, ..., J$. We then define the ensemble mean and the corresponding fluctuation by
\begin{gather}
\qquad \left\langle u \right\rangle ^n : =\frac{1}{J}\sum_{j=1}^{J}(3u_{j}^{n}-3u_{j}^{n-1}+u_j^{n-2}),\tag{mean}\label{Enmean}\\
 u_j^{\prime n}: =3u_{j}^{n}-3u_{j}^{n-1}+u_j^{n-2}-\left\langle u\right\rangle ^n.\tag{fluctuation}\label{Enfluc}
\end{gather}
$3u^{n}-3u^{n-1}+u^{n-2}$ is a third order extrapolation of $u^{n+1}$. Taking $\gamma=\frac{1}{2}$, we consider the following blended BDF for discretization of the time derivative of velocity $u.$

\begin{gather}
D_{\frac{1}{2}}(u_t^{n+1})= 
\frac{10u^{n+1}-15u^{n}+6u^{n-1}-u^{n-2}}{6\Delta t}.\nonumber
\end{gather}

Suppressing the spacial discretization, the second order ensemble method we study reads: for $j=1, ..., J$, given $u_j^0$, $u_j^1$ and $u_j^2$, find $u_j^{n+1}$ satisfying
\begin{gather}
\frac{10u_{j}^{n+1}-15u_{j}^{n}+6u_j^{n-1}-u_j^{n-2}}{6\Delta t}+\left\langle u\right\rangle^{n}\cdot\nabla u_{j}^{n+1}\tag{\text{En-BlendedBDF}}\label{EnBDF2AB2}\\
\qquad\qquad\qquad\qquad\qquad+u_j^{\prime n}\cdot\nabla \left(3u_{j}^{n}-3u_j^{n-1}+u_j^{n-2}\right)+\nabla p_{j}^{n+1}-\nu\Delta u_{j}^{n+1}=f_{j}^{n+1}\text{, \ }\nonumber
\\
\nabla\cdot u_{j}^{n+1}=0.\nonumber
\end{gather}

This is a four-level method. $u^0_j$ comes from given initial conditions of the problem. We need to obtain $u^1_j$ through some one-step method, such as Crank-Nicolson method. To get $u_j^2$, one can either use a one-step method or two-step method. The errors in these first few steps will affect the overall convergence rate of the method and thus they all need to be second order methods. We emphasize here that the timestepping schemes studied in \cite{NSMH03} are applied to single Navier-Stokes equations, while the ensemble timestepping method we study in this paper deals with multiple Navier-Stokes equations, for which the fluctuation-induced instability has to be taken into account.

The rest of the paper is organized as follows. In Section 2, we present the notation that will be used throughout the work, and the finite element formulation of the proposed method. In the third section, the long time stability of the method is proved under a CFL-like condition. In Section 4, we first provide upper bounds for the consistency error and then give a comprehensive error analysis for the fully discretized method. Numerical experiments and results are presented in Section 5 to confirm theoretical analysis. Finally, in Section 6, we state some concluding remarks.
}
\section{Notation and preliminaries}

Throughout this paper the $L^2(\Omega)$ norm of scalars, vectors, and tensors will be denoted by $\Vert \cdot\Vert$ with the usual $L^2$ inner product denoted by $(\cdot, \cdot)$.  $H^{k}(\Omega)$ is the Sobolev space $%
W_{2}^{k}(\Omega)$, with norm $\Vert\cdot\Vert_{k}$. 
Let $X,Q$ denote the velocity, pressure, and divergence free velocity spaces:
\begin{align*}
X:&=H_{0}^{1}(\Omega )^{d}=\left\lbrace v\in L^2(\Omega)^d: \nabla v\in L^2(\Omega)^{d\times d}\text{ and }v=0 \text{ on } \partial \Omega\right\rbrace,\\
Q : &=L_{0}^{2}(\Omega )=\left\lbrace q\in L^2(\Omega): \int_{\Omega} q \text{ }dx=0\right\rbrace,\\
V : &=\left\lbrace v\in X:  (\nabla\cdot v, q)=0, \forall q\in Q \right\rbrace.
\end{align*}
For $\forall u, v, w \in X$, we define the usual skew symmetric trilinear form $$b^{\ast }(u,v,w):=\frac{1}{2}
(u\cdot \nabla v,w)-\frac{1}{2}(u\cdot \nabla w,v),$$
which satisfies 
\begin{gather}
|b^{\ast }(u,v,w)|\leq C(\Omega )\Vert \nabla u\Vert\Vert\nabla
v\Vert\Vert\nabla w\Vert,\label{bd:0tri}\\
|b^{\ast }(u,v,w)|\leq C(\Omega )\Vert u \Vert^{1/2}\Vert \nabla u\Vert^{1/2}\Vert\nabla
v\Vert\Vert\nabla w\Vert,\label{bd:tri}\\
|b^{\ast }(u,v,w)|\leq C(\Omega )\Vert \nabla u\Vert\Vert\nabla
v\Vert\Vert\nabla w\Vert^{1/2}\Vert w\Vert^{1/2}.\label{bd:2tri}
\end{gather}
For $\forall u_h, v_h, w_h \in X_h$, we have \cite{JL14}
\begin{gather}
b^{\ast}(u_h, v_h, w_h)= \int_{\Omega}u_h\cdot\nabla v_h\cdot w_h dx +\frac{1}{2}\int_{\Omega} (\nabla\cdot u_h)(v_h\cdot w_h) dx.
\end{gather} 
\begin{lemma}
For $\forall u_h, v_h, w_h \in X_h$,
\begin{gather}
b^{\ast}(u_h, v_h, w_h)\leq \Vert u_h\Vert_{L^4}\Vert \nabla v_h\Vert\Vert w_h \Vert_{L^4}+C\Vert \nabla \cdot
u_h\Vert_{L^4}\Vert \nabla v_h\Vert\Vert w_h \Vert_{L^4}.
\end{gather}
\end{lemma}
\begin{proof}
By H$\ddot{o}$lder's inequality,
\begin{align*}
&\int_{\Omega}u_h\cdot\nabla v_h\cdot w_h dx \\
&\leq \left(\int_{\Omega} \vert u_h\cdot \nabla v_h\vert^{4/3}dx \right)^{3/4}\cdot\left(\int_{\Omega}\vert w_h\vert^4dx\right)^{1/4}\\
&\leq \left( \left(\int_{\Omega} \left(\vert u_h\vert^{4/3}\right)^{3} dx\right)^{1/3} \cdot\left(\int_{\Omega} \left(\vert \nabla v_h\vert^{4/3}\right)^{3/2} dx\right)^{2/3}  \right)^{3/4}\cdot\left(\int_{\Omega}\vert w_h\vert^4dx\right)^{1/4}\\
&\leq \Vert u_h \Vert_{L^4}\Vert \nabla v_h\Vert\Vert w_h\Vert_{L^4}
\end{align*} 
Similarly, we have
\begin{align*}
&\int_{\Omega} (\nabla \cdot u_h)( v_h\cdot w_h) dx \\
&\leq \left(\int_{\Omega}\vert \nabla \cdot u_h\vert^4dx\right)^{1/4}\cdot\left(\int_{\Omega} \vert v_h\cdot w_h\vert^{4/3}dx \right)^{3/4}\\
&\leq\left(\int_{\Omega}\vert \nabla \cdot u_h\vert^4dx\right)^{1/4}\cdot \left( \left(\int_{\Omega} \left(\vert w_h\vert^{4/3}\right)^{3} dx\right)^{1/3} \cdot\left(\int_{\Omega} \left(\vert v_h\vert^{4/3}\right)^{3/2} dx\right)^{2/3}  \right)^{3/4}\\
&\leq \Vert \nabla \cdot  u_h \Vert_{L^4}\Vert  v_h\Vert\Vert w_h\Vert_{L^4}\\
&\leq C\Vert \nabla \cdot  u_h \Vert_{L^4}\Vert \nabla  v_h\Vert\Vert w_h\Vert_{L^4}
\end{align*} 

\end{proof}

In the two-dimensional space $(d=2)$, Ladyzhenskaya's inequality is 
\begin{gather}\label{lady}
\Vert u\Vert_{L^4}    \leq C \Vert u\Vert^{1/2}\Vert \nabla u\Vert^{1/2}.
\end{gather}
The norm on the dual space of $X$ is defined by
\begin{equation*}
\Vert f\Vert _{-1}=\sup_{0\neq v\in X}\frac{(f,v)}{\Vert
\nabla v\Vert }\text{ .}
\end{equation*}
We denote conforming velocity, pressure finite element spaces based on an edge to edge
triangulation ($d=2$) or tetrahedralization ($d=3$) of $\Omega $ with maximum element diameter $h$ by 
\begin{equation*}
X_{h}\subset X\text{ }\text{, }Q_{h}\subset
Q.
\end{equation*}%
We also assume the finite element spaces ($X_{h}$, $Q_{h}$) satisfy the usual discrete inf-sup /$ LBB$
condition for stability of the discrete pressure, see \cite{G89} for more on this condition. Taylor-Hood elements, e.g., \cite{BS08}, \cite{G89}, are one such choice used in the tests in Section $6$. The discretely divergence free
subspace of $X_{h}$ is 
\begin{equation*}
V_{h} :\text{}=\{v_{h}\in X_{h}:(\nabla \cdot v_{h},q_{h})=0\text{ , }%
\forall q_{h}\in Q_{h}\}.
\end{equation*}%
We assume further that the finite element spaces satisfy the inverse inequality (typical for
quasi-uniform meshes, e.g., \cite{BS08}), for all $v_{h}\in X_{h}$, 
\begin{align}
h\Vert \nabla v_{h}\Vert & \leq C\Vert
v_{h}\Vert .\label{bd:inv}
\end{align}

The fully discrete method is: given $u_{j,h}^{n-2}, u_{j,h}^{n-1}, u_{j,h}^{n}$, find
$u_{j,h}^{n+1}\in X_{h}$, $p_{j,h}^{n+1}\in Q_{h}$ satisfying
\begin{center}
\begin{gather}
\left(\frac{10u_{j,h}^{n+1}-15u_{j,h}^{n}+6u_{j,h}^{n-1}
-u^{n-2}_{j,h}}{6\Delta t},v_{h}\right)+b^{\ast}\left(\left\langle u_{h}\right\rangle^{n},u_{j,h}^{n+1},v_{h}\right)\label{EnBDF2AB2-h}\\
+b^{\ast}\left( u_{j,h}^{\prime n},3u_{j,h}^{n}-3u_{j,h}^{n-1}+u_{j,h}^{n-2}
,v_{h}\right)
-\left(p_{j,h}^{n+1},\nabla\cdot v_{h}\right)\nonumber\\
+\nu\left(\nabla u_{j,h}^{n+1},\nabla
v_{h}\right)=\left(f_{j}^{n+1},v_{h}\right)\text{, }\qquad\forall v_{h}\in X_{h}%
,\nonumber\\
\left(\nabla\cdot u_{j,h}^{n+1},q_{h}\right)=0,\qquad\forall q_{h}\in Q_{h}.\nonumber
\end{gather}
\end{center}
\section{Stability of the method}
In this section, we prove (En-BlendedBDF) is long time, nonlinearly stable under a CFL-like time step condition.
\begin{theorem}
[Stability of (En-BlendedBDF)]\label{th:1} Consider the method
(\ref{EnBDF2AB2-h}) with a standard
spacial discretization with mesh size $h$. Suppose the following time step conditions hold: 
\begin{align}
C \frac{\Delta t}{\nu h}
\Vert \nabla u_{j,h}^{\prime n}\Vert^2\leq 1,  \qquad j= 1, ..., J.\label{ineq:CFL} 
\end{align}
Then, for
any
$N > 2$
\begin{gather}
\frac{1}{12}\|u_{j,h}^{N}\|^{2}+\frac{1}{12}\| 3u_{j,h}
^{N}-u_{j,h}^{N-1}\|^{2}+\frac{1}{12}\Vert 3u_{j,h}^{N}-3u_{j,h}^{N-1}+u_{j,h}^{N-2}\Vert^2\label{in:EnB}\\
+\frac{1}{24}\sum_{n=2}^{N-1}\Vert u_{j,h}^{n+1}-3u_{j,h}^{n}+3u_{j,h}^{n-1}-u_{j,h}^{n-2}\Vert^{2}
+\frac{\Delta t }{4}\sum_{n=2}^{N-1}\nu\Vert \nabla u_{j,h}^{n+1}\Vert^{2}\nonumber\\
\leq\sum_{n=2}^{N-1}\frac{\Delta t}{\nu}\|f_{j}^{n+1}\|_{-1}^{2}+ \frac{1}%
{12}\|u_{j,h}^{2}\|^{2}+\frac{1}{12}\| 3u_{j,h}
^{2}-u_{j,h}^{1}\|^{2}+\frac{1}{12}\Vert 3u_{j,h}^{2}-3u_{j,h}^{1}+u_{j,h}^{0}\Vert^2\text{
.}\nonumber
\end{gather}

\end{theorem}

\begin{proof}
Set $v_{h}=u_{j,h}^{n+1}$ in (\ref{EnBDF2AB2-h}), multiply through by $\Delta t$ and apply Young's inequality to the right hand side. This gives

\begin{gather}
\frac{1}{12}\left(\Vert u_{j,h}^{n+1}\Vert^{2}-\Vert u_{j,h}^{n}\Vert^{2}\right)+\frac{1}{12}\left(\Vert 3u_{j,h}^{n+1}-u_{j,h}^{n}\Vert^2-\Vert 3u_{j,h}^{n}-u_{j,h}^{n-1}\Vert^2\right)\label{qq}\\
\frac{1}{12}\left(\Vert 3 u^{n+1}_{j,h} -3 u^{n}_{j,h}+u^{n-1}_{j,h}\Vert^2-\Vert 3u^n_{j,h}-3u^{n-1}_{j,h}+u^{n-2}_{j,h}\Vert^2\right)\nonumber\\
+\frac{1}%
{12}\Vert u_{j,h}^{n+1}-3u_{j,h}^{n}+3u_{j,h}^{n-1}-u^{n-2}_{j,h}\Vert^{2}+ \Delta t b^{*}\left(  u_{j,h}^{\prime n},  3u_{j,h}^{n}-3u_{j,h}^{n-1}+u_{j,h}^{n-2},u_{j,h}^{n+1}\right)\nonumber\\+
\nu\Delta t \Vert \nabla u_{j,h}^{n+1}\Vert^{2}
\leq\frac{\nu\Delta t}{4} \Vert\nabla u_{j,h}^{n+1} \Vert^{2}+ \frac{\Delta t}{
\nu} \Vert f_{j}^{n+1}\Vert_{-1}^{2}\text{ .}\nonumber
\end{gather}

\noindent Next, we bound the remaining trilinear term using \eqref{bd:2tri}, \eqref{bd:inv} and Young's inequality.

\begin{align}
&\Delta t b^{*}\left( u_{j,h}^{\prime n},  3u_{j,h}^{n}-3u_{j,h}^{n-1}
+u_{j,h}^{n-2},u_{j,h}^{n+1}\right)\\
&=\Delta t b^{*}\left(  u_{j,h}^{\prime n}, u_{j,h}^{n+1}, u_{j,h}^{n+1}-3u_{j,h}^{n}+3u_{j,h}^{n-1}-u_{j,h}^{n-2}\right)\nonumber\\
&\leq C\Delta t\Vert\nabla u_{j,h}^{\prime n}\Vert\Vert\nabla u_{j,h}^{n+1}\Vert\Vert \nabla (u_{j,h}^{n+1}-3u_{j,h}^{n}+3u_{j,h}^{n-1}-u_{j,h}^{n-2})\Vert^{\frac{1}{2}}\Vert u_{j,h}^{n+1}-3u_{j,h}^{n}+3u_{j,h}^{n-1}-u_{j,h}^{n-2}\Vert^{\frac{1}{2}}\nonumber\\
&\leq C\Delta t h^{-\frac{1}{2}%
}\Vert\nabla u_{j,h}^{\prime n}\Vert\Vert\nabla u_{j,h}^{n+1}\Vert\Vert u_{j,h}^{n+1}-3u_{j,h}^{n}+3u_{j,h}^{n-1}-u_{j,h}^{n-2}\Vert\nonumber\\
&\leq  C \frac{\Delta t^{2}}{h}
\Vert \nabla u_{j,h}^{\prime n}\Vert^{2} \Vert\nabla u_{j,h}^{n+1}\Vert^{2}
+
\frac{1}{24}\Vert u_{j,h}^{n+1}-3u_{j,h}^{n}+3u_{j,h}^{n-1}-u_{j,h}^{n-2}\Vert^{2}\text{ .}\nonumber
\end{align}

\noindent With this bound, combining like terms, \eqref{qq} becomes

\begin{gather}
\frac{1}{12}\left(\Vert u_{j,h}^{n+1}\Vert^{2}-\Vert u_{j,h}^{n}\Vert^{2}\right)+\frac{1}{12}\left(\Vert 3u_{j,h}^{n+1}-u_{j,h}^{n}\Vert^2-\Vert 3u_{j,h}^{n}-u_{j,h}^{n-1}\Vert^2\right)\label{ineq:stable01}\\
\frac{1}{12}\left(\Vert 3 u^{n+1}_{j,h} -3 u^{n}_{j,h}+u^{n-1}_{j,h}\Vert^2-\Vert 3u^n_{j,h}-3u^{n-1}_{j,h}+u^{n-2}_{j,h}\Vert^2\right)\nonumber\\
+\frac{\nu\Delta t}{4} \Vert \nabla u_{j,h}^{n+1}\Vert^{2}+
\frac{\nu\Delta t}{2}\left(1-C \frac{\Delta t}{\nu h}
\Vert \nabla u_{j,h}^{\prime n}\Vert^{2}\right) \Vert \nabla u_{j,h}^{n+1}\Vert^{2}\nonumber\\
+\frac{1}%
{24}\Vert u_{j,h}^{n+1}-3u_{j,h}^{n}+3u_{j,h}^{n-1}-u_{j,h}^{n-2}\Vert^{2}\leq\frac{\Delta t}{ \nu} \Vert f_{j}^{n+1}\Vert_{-1}^{2}\text{ .}\nonumber
\end{gather}

\noindent With the time step restriction (\ref{ineq:CFL}) assumed, we have

\begin{gather*}
\frac{\nu\Delta t}{2}\left(1-C \frac{\Delta t}{\nu h}
\Vert \nabla u_{j,h}^{\prime n}\Vert^{2}\right) \Vert \nabla u_{j,h}^{n+1}\Vert^{2}\geq0\text{ .}%
\end{gather*}

\noindent Equation (\ref{ineq:stable01}) reduces to%

\begin{gather}
\frac{1}{12}\left(\Vert u_{j,h}^{n+1}\Vert^{2}-\Vert u_{j,h}^{n}\Vert^{2}\right)+\frac{1}{12}\left(\Vert 3u_{j,h}^{n+1}-u_{j,h}^{n}\Vert^2-\Vert 3u_{j,h}^{n}-u_{j,h}^{n-1}\Vert^2\right)\label{ineq:stable02}\\
\frac{1}{12}\left(\Vert 3 u^{n+1}_{j,h} -3 u^{n}_{j,h}+u^{n-1}_{j,h}\Vert^2-\Vert 3u^n_{j,h}-3u^{n-1}_{j,h}+u^{n-2}_{j,h}\Vert^2\right)\nonumber\\
+\frac{\nu\Delta t}{4} \Vert \nabla u_{j,h}^{n+1}\Vert^{2}+\frac{1}{24}\Vert u_{j,h}^{n+1}-3u_{j,h}^{n}+3u_{j,h}^{n-1}-u_{j,h}^{n-2}\Vert^{2}\leq\frac{\Delta t}{ \nu} \Vert f_{j}^{n+1}\Vert_{-1}^{2}\text{ .}\nonumber
\end{gather}

\noindent Summing up (\ref{ineq:stable02}) from $n=2$ to $n=N-1$ results in \eqref{in:EnB}.
\end{proof}

\begin{remark} 
This time step condition seems very restrictive especially for high Reynolds number flows. However, it is shown in our numerical tests that this condition can be significantly weakened by adding grad-div stabilization, i.e., $\gamma (\nabla \cdot u^{n+1}_{j,h}, \nabla \cdot v_h) $. The grad-div stabilization is well known to help improve mass conservation and relax the effect of the pressure error on the velocity error, \cite{GLOS05}, \cite{O02}. 
\end{remark}

\subsection{An improved timestep condition for two-dimensional domains}

For two dimensional domains, there are better embedding estimates which can lead to improvements on the timestep restriction. In this section we give one such example by making use of the 2d version of Ladyzhenskaya's inequality \eqref{lady}. We prove (En-BlendedBDF) is long time, nonlinearly stable under a much less restrictive timestep condition \eqref{ineq:CFL2d}. If pointwise divergence free elements (e.g., Scott-Vogelius elements, \cite{CELR11}) are used, this 2d timestep restriction can be further relaxed. 

\begin{theorem}\label{th:2d} Consider the method
(\ref{EnBDF2AB2-h}) with a standard
spacial discretization with mesh size $h$. Suppose the computational domain is in the two-dimensional space $(d=2)$ and the following timestep conditions hold: 
\begin{align}
C \frac{\Delta t}{\nu h}
(\Vert u_{j,h}^{\prime n}\Vert+\Vert \nabla\cdot u_{j,h}^{\prime n}\Vert)^2\leq 1,  \qquad j= 1, ..., J.\label{ineq:CFL2d} 
\end{align}
Then, for
any
$N > 2$
\begin{gather}
\frac{1}{12}\|u_{j,h}^{N}\|^{2}+\frac{1}{12}\| 3u_{j,h}
^{N}-u_{j,h}^{N-1}\|^{2}+\frac{1}{12}\Vert 3u_{j,h}^{N}-3u_{j,h}^{N-1}+u_{j,h}^{N-2}\Vert^2\label{in:EnB}\\
+\frac{1}{24}\sum_{n=2}^{N-1}\Vert u_{j,h}^{n+1}-3u_{j,h}^{n}+3u_{j,h}^{n-1}-u_{j,h}^{n-2}\Vert^{2}
+\frac{\Delta t }{4}\sum_{n=2}^{N-1}\nu\Vert \nabla u_{j,h}^{n+1}\Vert^{2}\nonumber\\
\leq\sum_{n=2}^{N-1}\frac{\Delta t}{\nu}\|f_{j}^{n+1}\|_{-1}^{2}+ \frac{1}%
{12}\|u_{j,h}^{2}\|^{2}+\frac{1}{12}\| 3u_{j,h}
^{2}-u_{j,h}^{1}\|^{2}+\frac{1}{12}\Vert 3u_{j,h}^{2}-3u_{j,h}^{1}+u_{j,h}^{0}\Vert^2\text{
.}\nonumber
\end{gather}

\end{theorem}

\begin{proof}
By lemma 1 and Ladyzhenskaya's inequality \eqref{lady}, in the two-dimensional space we have the following bound on the nonlinear term.
\begin{align}
&\Delta t b^{*}\left( u_{j,h}^{\prime n},  3u_{j,h}^{n}-3u_{j,h}^{n-1}
+u_{j,h}^{n-2},u_{j,h}^{n+1}\right)\\
&=\Delta t b^{*}\left(  u_{j,h}^{\prime n}, u_{j,h}^{n+1}, u_{j,h}^{n+1}-3u_{j,h}^{n}+3u_{j,h}^{n-1}-u_{j,h}^{n-2}\right)\nonumber\\
&\leq \Delta t\Vert u_{j,h}^{\prime n}\Vert_{L^4}\Vert \nabla u_{j,h}^{n+1}\Vert\Vert  u_{j,h}^{n+1}-3u_{j,h}^{n}+3u_{j,h}^{n-1}-u_{j,h}^{n-2}\Vert_{L^4}\nonumber\\
&\quad+C\Delta t\Vert\nabla\cdot u_{j,h}^{\prime n}\Vert_{L^4}\Vert \nabla  u_{j,h}^{n+1} \Vert \Vert u_{j,h}^{n+1}-3u_{j,h}^{n}+3u_{j,h}^{n-1} -u_{j,h}^{n-2} \Vert_{L^4}\nonumber\\
&\leq C\Delta t(\Vert u_{j,h}^{\prime n}\Vert_{L^4}+\Vert\nabla \cdot u_{j,h}^{\prime n}\Vert_{L^4})\Vert \nabla u_{j,h}^{n+1}\Vert\cdot\nonumber\\
&\quad\Vert \nabla (u_{j,h}^{n+1}-3u_{j,h}^{n}+3u_{j,h}^{n-1}-u_{j,h}^{n-2})\Vert^{\frac{1}{2}}\Vert u_{j,h}^{n+1}-3u_{j,h}^{n}+3u_{j,h}^{n-1}-u_{j,h}^{n-2}\Vert^{\frac{1}{2}}\nonumber\\
&\leq C\Delta t h^{-\frac{1}{2}
}(\Vert u_{j,h}^{\prime n}\Vert_{L^4}+\Vert\nabla \cdot u_{j,h}^{\prime n}\Vert_{L^4})\Vert\nabla u_{j,h}^{n+1}\Vert\Vert u_{j,h}^{n+1}-3u_{j,h}^{n}+3u_{j,h}^{n-1}-u_{j,h}^{n-2}\Vert\nonumber\\
&\leq  C \frac{\Delta t^{2}}{h}
(\Vert u_{j,h}^{\prime n}\Vert_{L^4}+\Vert\nabla \cdot u_{j,h}^{\prime n}\Vert_{L^4})^2 \Vert\nabla u_{j,h}^{n+1}\Vert^{2}
+
\frac{1}{24}\Vert u_{j,h}^{n+1}-3u_{j,h}^{n}+3u_{j,h}^{n-1}-u_{j,h}^{n-2}\Vert^{2}\text{ .}\nonumber
\end{align}

\noindent Thus, \eqref{qq} reduces to

\begin{gather}
\frac{1}{12}\left(\Vert u_{j,h}^{n+1}\Vert^{2}-\Vert u_{j,h}^{n}\Vert^{2}\right)+\frac{1}{12}\left(\Vert 3u_{j,h}^{n+1}-u_{j,h}^{n}\Vert^2-\Vert 3u_{j,h}^{n}-u_{j,h}^{n-1}\Vert^2\right)\label{2d1}\\
\frac{1}{12}\left(\Vert 3 u^{n+1}_{j,h} -3 u^{n}_{j,h}+u^{n-1}_{j,h}\Vert^2-\Vert 3u^n_{j,h}-3u^{n-1}_{j,h}+u^{n-2}_{j,h}\Vert^2\right)+\frac{\nu\Delta t}{4} \Vert \nabla u_{j,h}^{n+1}\Vert^{2}\nonumber\\
+
\frac{\nu\Delta t}{2}\left(1-C \frac{\Delta t}{\nu h}
(\Vert u_{j,h}^{\prime n}\Vert_{L^4}+\Vert\nabla \cdot u_{j,h}^{\prime n}\Vert_{L^4})^2\right) \Vert \nabla u_{j,h}^{n+1}\Vert^{2}\nonumber\\
+\frac{1}%
{24}\Vert u_{j,h}^{n+1}-3u_{j,h}^{n}+3u_{j,h}^{n-1}-u_{j,h}^{n-2}\Vert^{2}\leq\frac{\Delta t}{ \nu} \Vert f_{j}^{n+1}\Vert_{-1}^{2}\text{ .}\nonumber
\end{gather}

Now if the timestep condition \eqref{ineq:CFL2d} holds, \eqref{in:EnB} follows by taking sum from $n=2$ to $n=N-1$.

\end{proof}

\section{Error Analysis}
In this section we will give full error analysis of (\text{En-BlendedBDF}). We first give a lemma on the estimate of the consistency error of the Blended BDF scheme. This result will be used in the error analysis for the fully discrete method. 
\begin{lemma}\label{lm} For any $u\in H^3(0,T;H^1(\Omega))$, the following inequalities hold.
\begin{gather}
\Big\Vert \frac{10u^{n+1}-15u^n+6u^{n-1}-u^{n-2}}{6\Delta t}-u_t^{n+1}\Big\Vert^2\leq \frac{7}{3}\Delta t^3 \left(\int_{t^{n-2}}^{t^{n+1}}\Vert \nabla u_{ttt}\Vert^2 dt\right),\label{cons1}\\
\Vert \nabla \left(u^{n+1}-3u^n+3u^{n-1}-u^{n-2}\right)\Vert^2\leq 9\Delta t^5 \left(\int_{t^{n-2}}^{t^{n+1}}\Vert \nabla u_{ttt}\Vert^2 dt\right).\label{cons2}
\end{gather}
\end{lemma}
\begin{proof} The technical proof is given in Appendix \ref{App:AppendixA}.

\end{proof}

For functions $v(x,t)$
defined on $\Omega\times(0,T)$, define $(1\leq m<\infty)$

\begin{equation*}
\Vert v\Vert _{\infty ,k}\text{ }:=EssSup_{[0,T]}\Vert v(\cdot, t )\Vert _{k}
\quad\text{  and  }\quad\Vert v\Vert _{m,k}\text{ }:=\left(\int_{0}^{T}\Vert v(\cdot, t
)\Vert _{k}^{m}dt\right)^{1/m}\text{ .}
\end{equation*}
We also introduce the following discrete
norms:
\begin{gather*}
\||v|\|_{\infty,k}\text{ }:=\max\limits_{0\leq n\leq N_{T}}\|v^{n}\|_{k}
\quad\text{ and }\quad
\||v|\|_{m,k}\text{ }:=\left(\sum_{n=0}^{N_{T}}\|v^{n}\|^{m}_{k} \Delta t\right)^{1/m}.
\end{gather*}
\noindent To analyze the rate of convergence of the approximation we assume that the
following regularity assumptions on the NSE
\begin{gather*}
u_{j} \in L^{\infty}\left(0,T;H^{1}(\Omega)\right)\cap H^{3}\left(0,T;H^{k+1}(\Omega)\right)\cap
H^{3}\left(0,T;H^{1}(\Omega)\right),\\
p_{j} \in L^{2}\left(0,T;H^{s+1}(\Omega)\right), \text{and }f_j \in L^{2}%
\left(0,T;L^{2}(\Omega)\right).
\end{gather*}
Assume $X_{h}$ and $Q_{h}$ satisfy the usual ($LBB^{h}$) condition, then the
method is equivalent to: \textit{for} $n=1,...,N_{T}-1$, \textit{find}
$u_{j,h}^{n+1} \in V_{h}$ \textit{such that}
\begin{gather}
\left(\frac{10u_{j,h}^{n+1}-15u_{j,h}^{n}+6u_{j,h}^{n-1}-u_{j,h}^{n-2}}{6\Delta t},v_{h}\right)+b^{\ast}\left(\left\langle u_{h}
\right\rangle^{n},u_{j,h}^{n+1},v_{h}\right)\label{EnEVBDF2AB2-h-1}\\
+b^{\ast}\left(u_{j,h}^{\prime n},3u_{j,h}^{n}-3u_{j,h}^{n-1}+u_{j,h}^{n-2}
,v_{h}\right)
+\nu \left(\nabla u_{j,h}^{n+1},\nabla
v_{h}\right)
=\left(f_{j}^{n+1},v_{h}\right)\text{ , }
\forall v_{h}\in V_{h}\nonumber.
\end{gather}

Let $e_{j}^{n}=u_{j}^{n}-u_{j,h}^{n}$ be the error between the true solution
and the approximate solution, then we have the following error estimates.

\begin{theorem}
[Convergence of (\text{En-BlendedBDF})]\label{th:errEnEVBDF2AB2} Consider the
method (\text{En-BlendedBDF}). If the following conditions hold
\begin{align}
C_e \frac{\Delta t}{\nu h}
\Vert \nabla u_{j,h}^{\prime n}\Vert^2\leq 1,  \qquad j= 1, ..., J,\label{err:CFL}
\end{align}
where $C_e$ is a constant that depends only on the domain and the minimum angle of the mesh and is independent of the timestep, then there is a positive constant $C$
independent of the mesh width and timestep such that
\begin{gather}
\frac{1}{2}\Vert e_{j}^{N}\Vert^{2}+\frac{1}{2}\|3e_{j}
^{N}-e_{j}^{N-1}\|^{2}+\frac{1}{2}\|3e_{j}
^{N}-3e_{j}^{N-1}+e_{j}^{N-2}\|^{2}\label{ineq:errlast2}\\
+\frac{1}{4}\sum_{n=2}^{N-1}\|e_{j}^{n+1}-3e_{j}^{n}
+3e_{j}^{n-1}-e_{j}^{n-2}\Vert^{2}
+\frac{3\nu\Delta t}{16}\Vert\nabla e_{j}^{N}\Vert^{2}\nonumber\\
+\nu\sum_{n=2}^{N-1}\Delta t
\Vert\nabla e_j^{n+1}\Vert^{2}+\frac{3\nu\Delta t}{8}
\left(\Vert\nabla e_{j}^{N}\Vert^{2}
+\Vert\nabla e_{j}^{N-1}\Vert^{2}\right)
+\frac{3\nu\Delta t}{16}\Vert\nabla e_{j}^{N-2}\Vert^{2}\nonumber\\
\leq exp\left(\frac{C T}{\nu^2}\right)\Bigg \lbrace\frac{1}{2}\Vert e_{j}^{2}\Vert^{2}+\frac{1}{2}\|3e_{j}
^{2}-e_{j}^{1}\|^{2}+\frac{1}{2}\|3e_{j}
^{2}-3e_{j}^{1}+e_{j}^{0}\|^{2}\nonumber\\
+\frac{3\nu\Delta t}{16}\Vert\nabla e_{j}^{2}\Vert^{2}
+\frac{3\nu\Delta t}{8}
\left(\Vert\nabla e_{j}^{2}\Vert^{2}
+\Vert\nabla e_{j}^{1}\Vert^{2}\right)
+\frac{3\nu\Delta t}{16}\Vert\nabla e_{j}^{0}\Vert^{2}\nonumber\\
+C\frac{h^{2k}}{\nu}\| |\nabla u_j|
\|^2_{\infty, 0}\| |u_j|\|^2_{2, k+1}+C\frac{\Delta t^6}{\nu}\| |\nabla u_{j, ttt}|\|^2_{2,0}
+C \frac{h^{2k}}{\nu}\||\nabla
u_{j}|\|^{2}_{2,k+1} \nonumber\\
+C\Delta t^4 h^{2k+1}\Vert \vert \nabla u_{j, ttt}\vert\Vert^2_{2,k}+C\Delta t^5 h  \| |\nabla u_{j, ttt}|\|^2_{2,0}
+C\frac{h^{2s+2}}{\nu}\||p_j|\|^{2}_{2,s+1} \nonumber\\
+Ch^{2k+2}\nu^{-1}\| |u_{j,t}|\|^{2}_{2,k+1}
+C\nu h^{2k}\| |\nabla u_{j} | \|_{2,k}^{2}+\frac{C{\Delta t}^{4}}{\nu}\| |
u_{j,ttt}| \| ^{2}_{2,0}\Bigg\rbrace\text{ .}\nonumber
\end{gather}

\end{theorem}

\begin{corollary}\label{cor1}
Under the assumptions of Theorem \ref{th:errEnEVBDF2AB2}, with $(X_{h},Q_{h}$) given by the P2-P1 Taylor-Hood approximation elements ($k=2, s=1$), i.e., $C^{0}$ piecewise quadratic velocity
space $X_{h}$ and $C^{0}$ piecewise linear pressure space $Q_{h}$,  we have the following error estimate
\begin{gather}
\frac{1}{2}\Vert e_{j}^{N}\Vert^{2}+\frac{1}{2}\|3e_{j}
^{N}-e_{j}^{N-1}\|^{2}+\frac{1}{2}\|3e_{j}
^{N}-3e_{j}^{N-1}+e_{j}^{N-2}\|^{2}+\frac{9\nu\Delta t}{16}\Vert\nabla e_{j}^{N}\Vert^{2}\label{ineq:errlast3}\\
+\frac{1}{4}\sum_{n=2}^{N-1}\|e_{j}^{n+1}-3e_{j}^{n}
+3e_{j}^{n-1}-e_{j}^{n-2}\Vert^{2}
+\frac{3\nu\Delta t}{8}\Vert\nabla e_{j}^{N-1}\Vert^{2}
+\frac{3\nu\Delta t}{16}\Vert\nabla e_{j}^{N-2}\Vert^{2}\nonumber\\
\leq C \Big(h^{4}+\Delta t^4+\frac{1}{2}\Vert e_{j}^{2}\Vert^{2}+\frac{1}{2}\|3e_{j}
^{2}-e_{j}^{1}\|^{2}+\frac{1}{2}\|3e_{j}
^{2}-3e_{j}^{1}+e_{j}^{0}\|^{2}\nonumber\\
+\frac{9\nu\Delta t}{16}\Vert\nabla e_{j}^{2}\Vert^{2}
+\frac{3\nu\Delta t}{8}\Vert\nabla e_{j}^{1}\Vert^{2}
+\frac{3\nu\Delta t}{16}\Vert\nabla e_{j}^{0}\Vert^{2}\Big)\text{ .}\nonumber
\end{gather}
\end{corollary}

{\allowdisplaybreaks
\begin{proof}

The true solution$(u_{j}, p_j)$ of the NSE  satisfies

\begin{gather}
\left(\frac{10u_{j}^{n+1}-15u_{j}^{n}+6u_j^{n-1}-u_{j,h}^{n-2}}{6 \Delta t}, v_{h}\right)+ b^{*}\left(u_{j}^{n+1},
u_{j}^{n+1}, v_{h}\right)\label{eq:convtrue}\\
+ \nu\left(\nabla u_{j}^{n+1}, \nabla v_{h}\right)- \left(p_{j}
^{n+1},\nabla\cdot v_{h}\right)
=\left(f_{j}^{n+1}, v_{h}\right)+ Intp\left(u_{j}^{n+1};v_{h}\right)\text{, } \text{ for all } v_{h}\in V_{h}\text{,}\nonumber
\end{gather}

\noindent where $Intp\left(u_{j}^{n+1};v_{h}\right)$ is defined as

\[
Intp\left(u_{j}^{n+1};v_{h}\right)=
\left(\frac{10u_{j}^{n+1}-15u_{j}^{n}+6u_j^{n-1}-u_{j,h}^{n-2}}{6\Delta t}-u_{j,t}(t^{n+1}),v_{h}\right)\text{ .}
\]

\noindent Let $e_{j}^{n}=u_{j}^{n}-u_{j,h}^{n}=\left(u_{j}^{n}-I_{h} u_{j}^{n}\right)+\left(I_{h} u_{j}%
^{n}-u_{j,h}^{n}\right)=\eta_{j}^{n}+\xi_{j,h}^{n}$, where $I_{h} u_{j}^{n} \in V_{h} $ is an interpolant of $u_{j}^{n}$ in
$V_{h}.$ Subtracting (\ref{EnEVBDF2AB2-h-1}) from (\ref{eq:convtrue}) gives%

\begin{gather}
\left(\frac{10\xi_{j,h}^{n+1}-15\xi_{j,h}^{n}+6\xi_{j,h}^{n-1}-\xi_{j,h}^{n-2}}{ 6\Delta t},v_{h}\right) +b^{*}\left(u_{j}^{n+1},u_{j}^{n+1},v_{h}\right)\label{eq:err}\\
+\nu\left(\nabla\xi
_{j,h}^{n+1},\nabla v_{h}\right)
-b^{*}\left(3u_{j,h}^{n}-3u_{j,h}^{n-1}+u_{j,h}^{n-2}-u_{j,h}^{\prime n},u_{j,h}^{n+1},v_{h}\right) \nonumber\\
-b^{*}\left(u_{j,h}^{\prime n},3u_{j,h}
^{n}-3u_{j,h}^{n-1}+u_{j,h}^{n-2},v_{h}\right)
-\left(p_{j}^{n+1},\nabla\cdot v_{h}\right)\nonumber\\
=-\left(\frac{10\eta_{j}^{n+1}-15\eta_{j}^{n}+6\eta_j^{n-1}-\eta_j{^{n-2}}}{6\Delta t},v_{h}\right) -\nu\left(\nabla\eta
_{j}^{n+1},\nabla v_{h}\right)
 +Intp\left(u_j^{n+1};v_{h}\right)\text{ .}\nonumber
\end{gather}

\noindent Set $v_{h}=\xi_{j,h}^{n+1}\in V_{h}$ , and rearrange the nonlinear terms, then
we have
\begin{gather}
\frac{1}{12\Delta t}\left(\Vert \xi_{j,h}^{n+1}\Vert^{2}-\Vert \xi_{j,h}^{n}\Vert^{2}\right)+\frac{1}{12\Delta t}\left(\Vert 3\xi_{j,h}^{n+1}-\xi_{j,h}^{n}\Vert^2-\Vert 3\xi
_{j,h}^{n}-\xi_{j,h}^{n-1}\Vert^{2}\right)\label{eq:err1}\\
\frac{1}{12\Delta t}\left(\Vert 3\xi_{j,h}^{n+1}-3\xi_{j,h}^n+\xi_{j,h}^{n-1}\Vert^2
-\Vert 3\xi_{j,h}^n-3\xi_{j,h}^{n-1}+\xi_{j,h}^{n-2}\Vert^2\right) \nonumber\\
+\frac{1}{12\Delta t}\|\xi_{j,h}^{n+1}-3\xi_{j,h}^{n}+3\xi_{j,h}^{n-1}-\xi_{j,h}^{n-2}\|^{2} +\nu \Vert \nabla \xi_{j,h}^{n+1}\Vert^2\nonumber\\
=-b^{*}\left(u_{j}^{n+1},u_{j}^{n+1},\xi_{j,h}^{n+1}\right)
+b^{*}\left(3u_{j,h}^{n}-3u_{j,h}^{n-1}+u_{j,h}^{n-2},u_{j,h}^{n+1},\xi_{j,h}^{n+1}\right)\nonumber\\
 +b^{*}\left(u_{j,h}^{\prime n},3u_{j,h}
^{n}-3u_{j,h}^{n-1}+u_{j,h}^{n-2}-u_{j,h}^{n+1},\xi_{j,h}^{n+1}\right)+\left(p_{j}^{n+1},\nabla\cdot\xi_{j,h}^{n+1}\right)\nonumber\\
-\left(\frac{10\eta_{j}^{n+1}-15\eta_{j}^{n}
+6\eta_j^{n-1}-u_{j,h}^{n-2}}{ 6\Delta t},\xi_{j,h}^{n+1}\right) -\nu
\left(\nabla\eta_{j}^{n+1},\nabla\xi_{j,h}^{n+1}\right)
 +Intp\left(u_j^{n+1};\xi_{j,h}
^{n+1}\right)\text{ .}\nonumber
\end{gather}

\noindent We first bound the nonlinear terms on the right hand side of equation \eqref{eq:err1}. Adding and subtracting $b^{*}(u_{j}^{n+1},u_{j,h}^{n+1}%
,\xi_{j,h}^{n+1})$, $b^{\ast}(3u^{n}_j-3u^{n-1}_j+u_{j}^{n-2}, u_{j,h}^{n+1}, \xi_{j,h}^{n+1})$ and $b^{*}(u_{j,h}^{\prime n},3u_{j}^{n}-3u_{j}^{n-1}+u_{j}^{n-2}-u_j^{n+1}
,\xi_{j,h}^{n+1})$ respectively, we rewrite the nonlinear terms as
\begin{gather}
-b^{*}\left(u_{j}^{n+1},u_{j}^{n+1},\xi_{j,h}^{n+1}\right)+b^{*}\left(3u_{j,h}^{n}
-3u_{j,h}^{n-1}+u_{j,h}^{n-2}
,u_{j,h}^{n+1},\xi_{j,h}^{n+1}\right)\label{eq:nonlinear}\\
+b^{*}\left(u_{j,h}^{\prime n},3u_{j,h}^{n}-3u_{j,h}^{n-1}+u_{j,h}^{n-2}-u^{n+1}_{j,h},\xi_{j,h}^{n+1}%
\right)\nonumber\\
=-b^{*}\left(u_{j}^{n+1},e_{j}^{n+1},\xi_{j,h}^{n+1}\right)-b^{*}\left(u_{j}^{n+1},u_{j,h}^{n+1}%
,\xi_{j,h}^{n+1}\right)\nonumber\\
+b^{*}\left(3u_{j,h}^{n}-3u_{j,h}^{n-1}+u_{j,h}^{n-2}
,u_{j,h}^{n+1},\xi_{j,h}^{n+1}\right)
+b^{*}\left(u_{j,h}^{\prime n},3u_{j,h}^{n}-3u_{j,h}^{n-1}+u_{j,h}^{n-2}-u^{n+1}_{j,h},\xi_{j,h}^{n+1}%
\right)\nonumber\\
=-b^{*}\left(u_{j}^{n+1},e_{j}^{n+1},\xi_{j,h}^{n+1}\right)-b^{*}\left(u_{j}^{n+1}-(3u_j^n-3u_j^{n-1}+u_{j,h}^{n-2}),u_{j,h}^{n+1}%
,\xi_{j,h}^{n+1}\right)\nonumber\\
-b^{*}\left(3e^{n}_j-3e^{n-1}_j+e_{j}^{n-2}
,u_{j,h}^{n+1},\xi_{j,h}^{n+1}\right)
+b^{*}\left(u_{j,h}^{\prime n},3u_{j,h}^{n}-3u_{j,h}^{n-1}+u_{j,h}^{n-2}-u^{n+1}_{j,h},\xi_{j,h}^{n+1}%
\right)\nonumber\\
=-b^{*}\left(u_{j}^{n+1},e_{j}^{n+1},\xi_{j,h}^{n+1}\right)-b^{*}\left(u_{j}^{n+1}-\left(3u_j^n-3u_j^{n-1}+u_j^{n-2}\right),u_{j,h}^{n+1}%
,\xi_{j,h}^{n+1}\right)\nonumber\\
-b^{*}\left(3e^{n}_j-3e^{n-1}_j+e^{n-2}_j
,u_{j,h}^{n+1},\xi_{j,h}^{n+1}\right)
-b^{*}\left(u_{j,h}^{\prime n},3e_j^{n}-3e_j^{n-1}+e_j^{n-2}-e_j^{n+1},\xi_{j,h}^{n+1}%
\right)\nonumber\\
+b^{*}\left(u_{j,h}^{\prime n},3u_{j}^{n}-3u_{j}^{n-1}+u_j^{n-2}-u^{n+1}_{j},\xi_{j,h}^{n+1}%
\right)\nonumber\\
=-b^{*}\left(u_{j}^{n+1},\eta_{j}^{n+1},\xi_{j,h}^{n+1}\right)-b^{*}\left(u_{j}^{n+1}
-(3u_{j}^{n}-3u_j^{n-1}+u_j^{n-2}),u_{j,h}^{n+1},\xi_{j,h}^{n+1}\right)\nonumber\\
-b^*\left(3\eta^n_j-3\eta^{n-1}_j+\eta^{n-2}_j, u_{j,h}^{n+1}, \xi_{j,h}^{n+1}\right)-b^*\left(3\xi^n_{j,h}-3\xi^{n-1}_{j,h}+\xi^{n-2}_{j,h}, u_{j,h}^{n+1}, \xi_{j,h}^{n+1}\right)\nonumber\\
-b^{*}\left(u_{j,h}^{\prime n},3\xi_{j,h}^{n}-3\xi_{j,h}^{n-1}+\xi_{j,h}^{n-2}-\xi_{j,h}^{n+1},\xi_{j,h}^{n+1}%
\right)-b^{*}\left(u_{j,h}^{\prime n},3\eta_j^{n}-3\eta_j^{n-1}+\eta_j^{n-2}-\eta_j^{n+1},\xi_{j,h}^{n+1}%
\right)\nonumber\\
+b^{*}\left(u_{j,h}^{\prime n},3u_{j}^{n}-3u_{j}^{n-1}+u_j^{n-2}-u^{n+1}_{j},\xi_{j,h}^{n+1}%
\right)\text{ .}\nonumber
\end{gather}

\noindent We estimate the nonlinear terms using \eqref{bd:0tri}, \eqref{bd:tri}, Lemma \ref{lm} and Young's inequality as follows.

\begin{gather}
b^{*}\left(u_{j}^{n+1},\eta_{j}^{n+1},\xi_{j,h}^{n+1}\right) \leq C\|\nabla u_{j}%
^{n+1}\|\|\nabla\eta_{j}^{n+1}\|\|\nabla\xi_{j,h}^{n+1}\|\\
\leq\frac{\nu}{64}\|\nabla\xi_{j,h}^{n+1}\|^{2}+C\nu^{-1}\|\nabla u_{j}%
^{n+1}\|^{2}\|\nabla\eta_{j}^{n+1}\|^{2}\text{ .}\nonumber
\end{gather}

\begin{gather}
b^{*}\left(u_{j}^{n+1}-\left(3u_{j}^{n}-3u_j^{n-1}+u_j^{n-2}\right),u_{j,h}^{n+1},\xi_{j,h}^{n+1}\right) \\
\leq C\|\nabla
\left(u_{j}^{n+1}-3u_{j}^{n}+3u_j^{n-1}-u_j^{n-2}\right)\|\|\nabla u_{j,h}^{n+1}\|\|\nabla\xi_{j,h}%
^{n+1}\|\nonumber\\
\leq\frac{\nu}{64}\|\nabla\xi_{j,h}^{n+1}\|^{2}+C\nu^{-1}\|\nabla \left(u_{j}
^{n+1}-3u_{j}^{n}+3u_j^{n-1}-u_j^{n-2}\right)\|^{2}\|\nabla u_{j,h}^{n+1}\|^{2}\nonumber\\
\leq\frac{\nu}{64}\|\nabla\xi_{j,h}^{n+1}\|^{2}+C\nu^{-1}\Delta t^5
\left(\int_{t^{n-2}}^{t^{n+1}}\| \nabla u_{j,ttt} \|^{2} dt\right)\|\nabla u_{j,h}%
^{n+1}\|^{2}\text{ . }\nonumber
\end{gather}
\begin{gather}
b^*\left(3\eta^n_j-3\eta^{n-1}_j+\eta^{n-2}_j, u_{j,h}^{n+1}, \xi_{j,h}^{n+1}\right) \\
\leq C\|\nabla
\left(3\eta^n_j-3\eta^{n-1}_j+\eta^{n-2}_j\right) \|\|\nabla u_{j,h}^{n+1}\|\|\nabla\xi_{j,h}%
^{n+1}\|\nonumber\\
\leq\frac{\nu}{64}\|\nabla\xi_{j,h}^{n+1}\|^{2}+C\nu^{-1}\left(\|\nabla\eta^n_j \|^{2}+\Vert \nabla\eta^{n-1}_j \|^{2}+\Vert \nabla\eta^{n-2}_j \|^{2}\right)\|\nabla u_{j,h}^{n+1}\|^{2} .\nonumber
\end{gather}

\begin{gather}
3b^*\left(\xi^n_{j,h}, u_{j,h}^{n+1}, \xi_{j,h}^{n+1}\right) \leq C\|
\nabla \xi^n_{j,h} \|^{\frac{1}{2}} \Vert \xi^n_{j,h}\Vert^{\frac{1}{2}}\|\nabla u_{j,h}^{n+1}\|\|\nabla\xi_{j,h}%
^{n+1}\|\\
\leq C\|
\nabla \xi^n_{j,h} \|^{\frac{1}{2}} \Vert \xi^n_{j,h}\Vert^{\frac{1}{2}}\|\nabla\xi_{j,h}%
^{n+1}\|
\leq C\left(\epsilon \|\nabla \xi_{j,h}^{n+1}\|^{2}+\frac{1}{\epsilon}\|\nabla \xi^n_{j,h} \|\|\xi^n_{j,h} \|\right) \nonumber\\
\leq C\left(\epsilon \|\nabla \xi_{j,h}^{n+1}\|^{2}+\frac{1}{\epsilon}\left(\delta \|\nabla \xi^n_{j,h} \|^2+\frac{1}{\delta}\|\xi^n_{j,h} \|^2\right)\right) \nonumber\\
\leq\left(\frac{\nu}{64} \|\nabla \xi_{j,h}^{n+1}\|^{2}+\frac{\nu}{32} \|\nabla \xi^n_{j,h} \|^2\right)+C\nu^{-3}\|\xi^n_{j,h} \|^2 .\nonumber
\end{gather}
Similarly,
\begin{gather}
3b^*\left(\xi^{n-1}_{j,h}, u_{j,h}^{n+1}, \xi_{j,h}^{n+1}\right) \leq C\|
\nabla \xi^{n-1}_{j,h} \|^{\frac{1}{2}} \Vert \xi^{n-1}_{j,h}\Vert^{\frac{1}{2}}\|\nabla u_{j,h}^{n+1}\|\|\nabla\xi_{j,h}%
^{n+1}\|\\
\leq C\|
\nabla \xi^{n-1}_{j,h} \|^{\frac{1}{2}} \Vert \xi^{n-1}_{j,h}\Vert^{\frac{1}{2}}\|\nabla\xi_{j,h}%
^{n+1}\|
\leq C\left(\epsilon \|\nabla \xi_{j,h}^{n+1}\|^{2}+\frac{1}{\epsilon}\|\nabla \xi^{n-1}_{j,h} \|\|\xi^{n-1}_{j,h} \|\right) \nonumber\\
\leq C\left(\epsilon \|\nabla \xi_{j,h}^{n+1}\|^{2}+\frac{1}{\epsilon}\left(\delta \|\nabla \xi^{n-1}_{j,h} \|^2+\frac{1}{\delta}\|\xi^{n-1}_{j,h} \|^2\right)\right) \nonumber\\
\leq \left(\frac{\nu}{64} \|\nabla \xi_{j,h}^{n+1}\|^{2}+\frac{\nu}{32} \|\nabla \xi^{n-1}_{j,h} \|^2\right)+C\nu^{-3}\|\xi^{n-1}_{j,h} \|^2 .\nonumber
\end{gather}
\begin{gather}
3b^*\left(\xi^{n-2}_{j,h}, u_{j,h}^{n+1}, \xi_{j,h}^{n+1}\right) \leq C\|
\nabla \xi^{n-2}_{j,h} \|^{\frac{1}{2}} \Vert \xi^{n-2}_{j,h}\Vert^{\frac{1}{2}}\|\nabla u_{j,h}^{n+1}\|\|\nabla\xi_{j,h}%
^{n+1}\|\\
\leq C\|
\nabla \xi^{n-2}_{j,h} \|^{\frac{1}{2}} \Vert \xi^{n-2}_{j,h}\Vert^{\frac{1}{2}}\|\nabla\xi_{j,h}%
^{n+1}\|
\leq C\left(\epsilon \|\nabla \xi_{j,h}^{n+1}\|^{2}+\frac{1}{\epsilon}\|\nabla \xi^{n-2}_{j,h} \|\|\xi^{n-2}_{j,h} \|\right) \nonumber\\
\leq C\left(\epsilon \|\nabla \xi_{j,h}^{n+1}\|^{2}+\frac{1}{\epsilon}\left(\delta \|\nabla \xi^{n-2}_{j,h} \|^2+\frac{1}{\delta}\|\xi^{n-2}_{j,h} \|^2\right)\right) \nonumber\\
\leq \left(\frac{\nu}{64} \|\nabla \xi_{j,h}^{n+1}\|^{2}+\frac{\nu}{32} \|\nabla \xi^{n-2}_{j,h} \|^2\right)+C\nu^{-3}\|\xi^{n-2}_{j,h} \|^2 .\nonumber
\end{gather}
By skew symmetry
\begin{gather}
b^{*}\left(u_{j,h}^{\prime n},3\xi_{j,h}^{n}-3\xi_{j,h}^{n-1}+\xi_{j,h}^{n-2}-\xi_{j,h}^{n+1},\xi_{j,h}^{n+1}%
\right)\nonumber\\
=-b^{*}\left(u_{j,h}^{\prime n},\xi_{j,h}^{n+1}-3\xi_{j,h}^{n}+3\xi_{j,h}^{n-1}-\xi_{j,h}^{n-2},\xi_{j,h}^{n+1}\right) \nonumber\\
=b^{*}\left(u_{j,h}^{\prime n},\xi_{j,h}^{n+1}, \xi_{j,h}^{n+1}-3\xi_{j,h}^{n}+3\xi_{j,h}^{n-1}
-\xi_{j,h}^{n-2}\right) .\nonumber
\end{gather}
Using \eqref{bd:2tri} and inverse inequality \eqref{bd:inv} gives
\begin{gather}
b^{*}\left(u_{j,h}^{\prime n},3\xi_{j,h}^{n}-3\xi_{j,h}^{n-1}+\xi_{j,h}^{n-2}-\xi_{j,h}^{n+1},\xi_{j,h}^{n+1}%
\right) \\
\leq C
 \|\nabla 
 u_{j,h}^{\prime n} \| \|\nabla \xi_{j,h}^{n+1}\|\| \nabla (\xi_{j,h}
^{n+1}-3\xi_{j,h}^{n}+3\xi_{j,h}^{n-1}-\xi_{j,h}^{n-2})\|^{\frac{1}{2}}\| \xi_{j,h}
^{n+1}-3\xi_{j,h}^{n}+3\xi_{j,h}^{n-1}-\xi_{j,h}^{n-2}\|^{\frac{1}{2}}\nonumber\\
 \leq C
 \|\nabla 
 u_{j,h}^{\prime n} \| \|\nabla \xi_{j,h}^{n+1}\|\left(h^{-\frac{1}{2}}\right)\| \xi_{j,h}
^{n+1}-3\xi_{j,h}^{n}+3\xi_{j,h}^{n-1}-\xi_{j,h}^{n-2}\|\nonumber\\
\leq\frac{1}{24\Delta t}\|\xi_{j,h}
^{n+1}-3\xi_{j,h}^{n}+3\xi_{j,h}^{n-1}-\xi_{j,h}^{n-2}\|^{2}+\frac{C_e}{32}\frac{\Delta t}{h}\|
 \nabla 
 u_{j,h}^{\prime n} \|^2\|\nabla \xi^{n+1}_{j,h} \|^{2}.\nonumber
\end{gather}

\begin{gather}
b^{*}\left(u_{j,h}^{\prime n},\eta_{j}^{n+1}-3\eta_{j}^{n}+3\eta_j^{n-1}-\eta_{j}^{n-2},\xi_{j,h}^{n+1}\right)\\
\leq C\|\nabla
u_{j,h}^{\prime n}\|\|\nabla\left(\eta_{j}^{n+1}-3\eta_{j}^{n}
+3\eta_j^{n-1}-\eta_j^{n-2}\right)\|\|\nabla\xi_{j,h}%
^{n+1}\|\nonumber\\
\leq\frac{\nu}{64}\|\nabla\xi_{j,h}^{n+1}\|^{2}+C\nu^{-1}\|\nabla u_{j,h}^{\prime n}\|^{2}\|\nabla\left(\eta_{j}^{n+1}-3\eta_{j}^{n}
+3\eta_j^{n-1}-\eta_j^{n-2}\right)\|^{2}\nonumber\\
\leq\frac{\nu}{64}\|\nabla\xi_{j,h}^{n+1}\|^{2}+\frac{C \Delta t^5}{\nu}\|\nabla
u_{j,h}^{\prime n}\|^{2}\left(\int_{t^{n-2}}^{t^{n+1}}\| \nabla \eta_{j,ttt}\|^{2} \text{ }dt\right)\text{
.}\nonumber
\end{gather}

\begin{gather}
b^{*}\left(u_{j,h}^{\prime n},u_{j}^{n+1}-3u_{j}^{n}+3u_j^{n-1}-u_j^{n-2},\xi_{j,h}^{n+1}\right)\\
\leq C\|\nabla
u_{j,h}^{\prime n}\|\|\nabla\left(u_{j}^{n+1}-3u_{j}^{n}+3u_j^{n-1}-u_j^{n-2}\right)\|\|\nabla\xi_{j,h}%
^{n+1}\|\nonumber\\
\leq\frac{\nu}{64}\|\nabla\xi_{j,h}^{n+1}\|^{2}+C\nu^{-1}\|\nabla u_{j,h}^{\prime n}\|^{2}\|\nabla\left(u_{j}^{n+1}-3u_{j}^{n}+3u_j^{n-1}-u_j^{n-2}\right)\|^{2}\nonumber\\
\leq\frac{\nu}{64}\|\nabla\xi_{j,h}^{n+1}\|^{2}+C \nu^{-1} \Delta t^5\|\nabla
u_{j,h}^{\prime n}\|^{2}\left(\int_{t^{n-2}}^{t^{n+1}}\| \nabla u_{j,ttt}\|^{2} \text{ }dt\right)\text{
.}\nonumber
\end{gather}

\noindent As $\xi_{j,h}^{n+1}\in V_{h}$ we have the following estimate for the pressure term

\begin{gather}
\left(p_{j}^{n+1},\nabla\cdot\xi_{j,h}^{n+1}\right)=\left(p_{j}^{n+1}-q_{j,h}^{n+1},
\nabla\cdot\xi_{j,h}^{n+1}\right)
\leq\|p_{j}^{n+1}-q_{j,h}^{n+1}\|\|\nabla\cdot\xi_{j,h}^{n+1}\|\nonumber\\
\leq\frac{\nu}{64}\|\nabla\xi_{j,h}^{n+1}\|^{2}+C \nu^{-1}\|p_{j}%
^{n+1}-q_{j,h}^{n+1}\|^{2} , \qquad\forall q_{j,h}^{n+1} \in Q_h \text{ .}\nonumber
\end{gather}

\noindent For the rest of the terms on the right hand side of \eqref{eq:err1} we have

\begin{gather}
\left(\frac{10\eta_{j}^{n+1}-15\eta_{j}^{n}+6\eta_j^{n-1}-
\eta_{j}^{n-2}}{ 6\Delta t},\xi_{j,h}^{n+1}\right)\\
\leq C\|\frac{10\eta_{j}^{n+1}-15\eta_{j}^{n}+6\eta_j^{n-1}-\eta_j^{n-2}}{ 6\Delta t}\| \|\nabla\xi
_{j,h}^{n+1}\|\nonumber\\
\leq C \nu^{-1}\|\frac{10\eta_{j}^{n+1}-15\eta_{j}^{n}+6\eta_j^{n-1}-\eta_j^{n-2}}{ 6\Delta t}\|^{2}+\frac{\nu}{64}\|\nabla\xi_{j,h}^{n+1}\|^{2}\nonumber\\
\leq C \nu^{-1}\|\frac{1}{\Delta t}\int_{t^{n-2}}^{t^{n+1}} \eta_{j,t} \text{ } dt
\|^2+\frac{\nu}{64}\|\nabla\xi_{j,h}^{n+1}\|^{2}
\leq\frac{C}{\nu\Delta t}\int_{t^{n-2}}^{t^{n+1}}\| \eta_{j,t}\|^{2}\text{ }
dt+\frac{\nu}{64}\|\nabla\xi_{j,h}^{n+1}\|^{2}\text{ ,}\nonumber
\end{gather}

\begin{gather}
\nu\left(\nabla\eta_{j}^{n+1},\nabla\xi_{j,h}^{n+1}\right) \leq\nu\|\nabla\eta_{j}%
^{n+1}\| \|\nabla\xi_{j,h}^{n+1}\|
\leq C\nu\|\nabla\eta_{j}^{n+1}\|^{2}+ \frac{\nu}{64}\|\nabla\xi_{j,h}%
^{n+1}\|^{2} \text{ ,}
\end{gather}

\noindent and 

\begin{gather}
Intp\left(u_{j}^{n+1};\xi_{j,h}^{n+1}\right)=\left(\frac{10u_{j}^{n+1}-15u_{j}^{n}
+6u_j^{n-1}-u_j^{n-2}}{6\Delta
t}-u_{j,t}(t^{n+1}),\xi_{j,h}^{n+1}\right)\nonumber\\
\leq C\|\frac{10u_{j}^{n+1}-15u_{j}^{n}+6u_j^{n-1}-u_j^{n-2}}{6\Delta t}-u_{j,t}(t^{n+1})\|
\|\nabla\xi_{j,h}^{n+1}\|\label{lastineq}\\
\leq\frac{\nu}{64}\|\nabla\xi_{j,h}^{n+1}\|^{2}+\frac{C}{\nu}\|\frac
{10u_{j}^{n+1}-15u_{j}^{n}+6u_j^{n-1}-u_j^{n-2}}{6\Delta t}-u_{j,t}(t^{n+1})\|^{2}\nonumber\\
\leq\frac{\nu}{64}\|\nabla\xi_{j,h}^{n+1}\|^{2}+\frac{C\Delta t^3}{\nu}%
\int_{t^{n-2}}^{t^{n+1}}\|u_{j,ttt}\|^{2} dt \text{ .}\nonumber
\end{gather}
\noindent Combining the above inequalities with \eqref{eq:err1} yields

\begin{gather}
\frac{1}{12\Delta t}\left(\Vert \xi_{j,h}^{n+1}\Vert^{2}-\Vert \xi_{j,h}^{n}\Vert^{2}\right)+\frac{1}{12\Delta t}\left(\Vert 3\xi_{j,h}^{n+1}-\xi_{j,h}^{n}\Vert^2-\Vert 3\xi
_{j,h}^{n}-\xi_{j,h}^{n-1}\Vert^{2}\right)\label{eq:err2}\\
\frac{1}{12\Delta t}\left(\Vert 3\xi_{j,h}^{n+1}-3\xi_{j,h}^n+\xi_{j,h}^{n-1}\Vert^2
-\Vert 3\xi_{j,h}^n-3\xi_{j,h}^{n-1}+\xi_{j,h}^{n-2}\Vert^2\right) \nonumber\\
+\frac{1}{24\Delta t}\|\xi_{j,h}^{n+1}-3\xi_{j,h}^{n}+3\xi_{j,h}^{n-1}-\xi_{j,h}^{n-2}\|^{2} +\frac{\nu}{32}
\left(\Vert\nabla\xi_{j,h}^{n+1}\Vert^{2}
-\Vert\nabla\xi_{j,h}^{n}\Vert^{2}\right)\nonumber\\
+\frac{\nu}{6}
\Vert\nabla\xi_{j,h}^{n+1}\Vert^{2}+\frac{\nu}{16}\left(
\left(\Vert\nabla\xi_{j,h}^{n+1}\Vert^{2}
+\Vert\nabla\xi_{j,h}^{n}\Vert^{2}\right)-
\left(\Vert\nabla\xi_{j,h}^{n}\Vert^{2}
+\Vert\nabla\xi_{j,h}^{n-1}\Vert^{2}\right)\right)
\nonumber\\
+\frac{\nu}{32}
\left(\Vert\nabla\xi_{j,h}^{n-1}\Vert^{2}
-\Vert\nabla\xi_{j,h}^{n-2}\Vert^{2}\right)
+\left(\frac{\nu}{32}-\frac{C_e}{32}\frac{\Delta t}{h}\| \nabla 
 u_{j,h}^{\prime n} \|^2\right)\|\nabla \xi^{n+1}_{j,h} \|^{2}\nonumber\\
\leq C\nu^{-3}\left(\|\xi^n_{j,h} \|^2+\|\xi^{n-1}_{j,h} \|^2+\Vert \xi_{j,h}^{n-2}\Vert^2\right)
+C\nu^{-1}\|\nabla u_{j}^{n+1}\|^{2}\|\nabla\eta_{j}^{n+1}\|^{2}\nonumber\\
+\frac{C\Delta t^5}{\nu}\left(\int_{t^{n-2}}^{t^{n+1}}\| \nabla u_{j,ttt} \|^{2}
dt\right)\|\nabla u_{j,h}^{n+1}\|^{2}
+\frac{C \Delta t^5}{\nu}\|\nabla
u_{j,h}^{\prime n}\|^{2}\left(\int_{t^{n-2}}^{t^{n+1}}\| \nabla \eta_{j,ttt}\|^{2} \text{ }dt\right)\nonumber\\
+C\nu^{-1}\left(\|\nabla\eta^n_j \|^{2}+\Vert \nabla\eta^{n-1}_j \|^{2}+\Vert \nabla \eta_j^{n-2}\Vert^2\right)\|\nabla u_{j,h}^{n+1}\|^{2}\nonumber\\
+\frac{C \Delta t^5}{\nu}\|\nabla
u_{j,h}^{\prime n}\|^{2}\left(\int_{t^{n-2}}^{t^{n+1}}\| \nabla u_{j,ttt}\|^{2} \text{ }dt\right)
+C \nu^{-1}\|p_{j}^{n+1}-q_{j,h}^{n+1}%
\|^{2}\nonumber\\
+\frac{C}{\nu\Delta t}\int_{t^{n-2}}^{t^{n+1}}\| \eta_{j,t}\|^{2}\text{ }
dt
+C\nu\|\nabla\eta_{j}^{n+1}\|^{2}
+\frac{C\Delta t^3}{\nu}%
\int_{t^{n-2}}^{t^{n+1}}\|u_{j,ttt}\|^{2} dt .\nonumber
\end{gather}

\noindent $(\frac{\nu}{32}-\frac{C_e}{32}\frac{\Delta t}{h}\| \nabla 
 u_{j,h}^{\prime n} \|^2)$ is nonnegative and thus can be eliminated from the left hand side of \eqref{eq:err2} if the timestep conditions in \eqref{err:CFL} hold. Taking the sum of (\ref{eq:err2}) from $n=2$ to $n=N-1$ and multiplying through by
$6\Delta t$, we obtain

\begin{gather}
\frac{1}{2}\Vert\xi_{j,h}^{N}\Vert^{2}+\frac{1}{2}\|3\xi_{j,h}
^{N}-\xi_{j,h}^{N-1}\|^{2}+\frac{1}{2}\|3\xi_{j,h}
^{N}-3\xi_{j,h}^{N-1}+\xi_{j,h}^{N-2}\|^{2}\label{ineq:err3}\\
+\frac{1}{4}\sum_{n=2}^{N-1}\|\xi_{j,h}^{n+1}-3\xi_{j,h}^{n}
+3\xi_{j,h}^{n-1}-\xi_{j,h}^{n-2}\Vert^{2}
+\frac{3\nu\Delta t}{16}\Vert\nabla\xi_{j,h}^{N}\Vert^{2}\nonumber\\
+\nu\sum_{n=2}^{N-1}\Delta t
\Vert\nabla\xi_{j,h}^{n+1}\Vert^{2}+\frac{3\nu\Delta t}{8}
\left(\Vert\nabla\xi_{j,h}^{N}\Vert^{2}
+\Vert\nabla\xi_{j,h}^{N-1}\Vert^{2}\right)
+\frac{3\nu\Delta t}{16}\Vert\nabla\xi_{j,h}^{N-2}\Vert^{2}\nonumber\\
\leq\frac{1}{2}\Vert\xi_{j,h}^{2}\Vert^{2}+\frac{1}{2}\|3\xi_{j,h}
^{2}-\xi_{j,h}^{1}\|^{2}+\frac{1}{2}\|3\xi_{j,h}
^{2}-3\xi_{j,h}^{1}+\xi_{j,h}^{0}\|^{2}
+\frac{3\nu\Delta t}{16}\Vert\nabla\xi_{j,h}^{2}\Vert^{2}\nonumber\\
+\frac{3\nu\Delta t}{8}
\left(\Vert\nabla\xi_{j,h}^{2}\Vert^{2}
+\Vert\nabla\xi_{j,h}^{1}\Vert^{2}\right)+\frac{3\nu\Delta t}{16}\Vert\nabla\xi_{j,h}^{0}\Vert^{2}
+\Delta t\sum_{n=0}^{N-1} C\nu^{-3}\|\xi^n_{j,h} \|^{2} \nonumber\\
+\Delta t\sum_{n=0}^{N-1} C\nu^{-1}\|\eta^n_{j} \|^{2} 
+\Delta t \sum_{n=2}^{N-1} \Bigg \lbrace C \nu
^{-1}\|\nabla u_{j}^{n+1}\|^{2}\|\nabla\eta_{j}^{n+1}\|^{2}\nonumber\\
+\frac{C\Delta
t^5}{\nu}\left(\int_{t^{n-2}}^{t^{n+1}}\| \nabla u_{j,ttt} \|^{2} \text{ }dt\right)
+C \Delta t^4h\left(\int_{t^{n-2}}^{t^{n+1}}\|
\nabla \eta_{j,ttt} \|^{2} dt\right)\nonumber\\
+C \Delta t^4h\left(\int_{t^{n-2}}^{t^{n+1}}\|
\nabla u_{j,ttt} \|^{2} dt\right)
+C \nu^{-1}\|p_{j}^{n+1}-q_{j,h}
^{n+1}\|^{2}\nonumber\\
+\frac{C}{\nu\Delta t}\int_{t^{n-2}}^{t^{n+1}}\|\eta_{j,t}\|^{2} dt 
+C\nu \|\nabla \eta_{j}^{n+1}\|^{2}+\frac{C{\Delta t^3}}{\nu}\int_{t^{n-2}
}^{t^{n+1}}\| u_{j,ttt}\| ^{2} \text{ }dt\Bigg\rbrace .\nonumber
\end{gather}

\noindent Applying interpolation inequalities to the above inequality gives

\begin{gather}
\frac{1}{2}\Vert\xi_{j,h}^{N}\Vert^{2}+\frac{1}{2}\|3\xi_{j,h}
^{N}-\xi_{j,h}^{N-1}\|^{2}+\frac{1}{2}\|3\xi_{j,h}
^{N}-3\xi_{j,h}^{N-1}+\xi_{j,h}^{N-2}\|^{2}\label{ineq:err3}\\
+\frac{1}{4}\sum_{n=2}^{N-1}\|\xi_{j,h}^{n+1}-3\xi_{j,h}^{n}
+3\xi_{j,h}^{n-1}-\xi_{j,h}^{n-2}\Vert^{2}
+\frac{3\nu\Delta t}{16}\Vert\nabla\xi_{j,h}^{N}\Vert^{2}\nonumber\\
+\nu\sum_{n=2}^{N-1}\Delta t
\Vert\nabla\xi_{j,h}^{n+1}\Vert^{2}+\frac{3\nu\Delta t}{8}
\left(\Vert\nabla\xi_{j,h}^{N}\Vert^{2}
+\Vert\nabla\xi_{j,h}^{N-1}\Vert^{2}\right)
+\frac{3\nu\Delta t}{16}\Vert\nabla\xi_{j,h}^{N-2}\Vert^{2}\nonumber\\
\leq\frac{1}{2}\Vert\xi_{j,h}^{2}\Vert^{2}+\frac{1}{2}\|3\xi_{j,h}
^{2}-\xi_{j,h}^{1}\|^{2}+\frac{1}{2}\|3\xi_{j,h}
^{2}-3\xi_{j,h}^{1}+\xi_{j,h}^{0}\|^{2}
+\frac{3\nu\Delta t}{16}\Vert\nabla\xi_{j,h}^{2}\Vert^{2}\nonumber\\
+\frac{3\nu\Delta t}{8}
\left(\Vert\nabla\xi_{j,h}^{2}\Vert^{2}
+\Vert\nabla\xi_{j,h}^{1}\Vert^{2}\right)
+\frac{3\nu\Delta t}{16}\Vert\nabla\xi_{j,h}^{0}\Vert^{2}
+\Delta t\sum_{n=0}^{N-1} C\nu^{-3}\|\xi^n_{j,h} \|^{2} \nonumber\\
+C\frac{h^{2k}}{\nu}\| |\nabla u_j|
\|^2_{\infty, 0}\| |u_j|\|^2_{2, k+1}+C\frac{\Delta t^6}{\nu}\| |\nabla u_{j, ttt}|\|^2_{2,0}
+C \frac{h^{2k}}{\nu}\||\nabla
u_{j}|\|^{2}_{2,k+1} \nonumber\\
+C\Delta t^4 h^{2k+1}\Vert \vert \nabla u_{j, ttt}\vert\Vert^2_{2,k}+C\Delta t^5 h  \| |\nabla u_{j, ttt}|\|^2_{2,0}
+C\frac{h^{2s+2}}{\nu}\||p_j|\|^{2}_{2,s+1} \nonumber\\
+Ch^{2k+2}\nu^{-1}\| |u_{j,t}|\|^{2}_{2,k+1}
+C\nu h^{2k}\| |\nabla u_{j} | \|_{2,k}^{2}+\frac{C{\Delta t}^{4}}{\nu}\| |
u_{j,ttt}| \| ^{2}_{2,0}.\nonumber
\end{gather}

\noindent Further applying the discrete Gronwall inequality (Girault and Raviart \cite{GR79}, p. 176) yields

\begin{gather}
\frac{1}{2}\Vert\xi_{j,h}^{N}\Vert^{2}+\frac{1}{2}\|3\xi_{j,h}
^{N}-\xi_{j,h}^{N-1}\|^{2}+\frac{1}{2}\|3\xi_{j,h}
^{N}-3\xi_{j,h}^{N-1}+\xi_{j,h}^{N-2}\|^{2}\label{ineq:errlast1}\\
+\frac{1}{4}\sum_{n=2}^{N-1}\|\xi_{j,h}^{n+1}-3\xi_{j,h}^{n}
+3\xi_{j,h}^{n-1}-\xi_{j,h}^{n-2}\Vert^{2}
+\frac{3\nu\Delta t}{16}\Vert\nabla\xi_{j,h}^{N}\Vert^{2}\nonumber\\
+\nu\sum_{n=2}^{N-1}\Delta t
\Vert\nabla\xi_{j,h}^{n+1}\Vert^{2}+\frac{3\nu\Delta t}{8}
\left(\Vert\nabla\xi_{j,h}^{N}\Vert^{2}
+\Vert\nabla\xi_{j,h}^{N-1}\Vert^{2}\right)
+\frac{3\nu\Delta t}{16}\Vert\nabla\xi_{j,h}^{N-2}\Vert^{2}\nonumber\\
\leq exp\left(\frac{C N\Delta t}{\nu^2}\right)\Bigg \lbrace\frac{1}{2}\Vert\xi_{j,h}^{2}\Vert^{2}+\frac{1}{2}\|3\xi_{j,h}
^{2}-\xi_{j,h}^{1}\|^{2}+\frac{1}{2}\|3\xi_{j,h}
^{2}-3\xi_{j,h}^{1}+\xi_{j,h}^{0}\|^{2}\nonumber\\
+\frac{3\nu\Delta t}{16}\Vert\nabla\xi_{j,h}^{2}\Vert^{2}
+\frac{3\nu\Delta t}{8}
\left(\Vert\nabla\xi_{j,h}^{2}\Vert^{2}
+\Vert\nabla\xi_{j,h}^{1}\Vert^{2}\right)
+\frac{3\nu\Delta t}{16}\Vert\nabla\xi_{j,h}^{0}\Vert^{2}\nonumber\\
+C\frac{h^{2k}}{\nu}\| |\nabla u_j|
\|^2_{\infty, 0}\| |u_j|\|^2_{2, k+1}+C\frac{\Delta t^6}{\nu}\| |\nabla u_{j, ttt}|\|^2_{2,0}
+C \frac{h^{2k}}{\nu}\||\nabla
u_{j}|\|^{2}_{2,k+1} \nonumber\\
+C\Delta t^4 h^{2k+1}\Vert \vert \nabla u_{j, ttt}\vert\Vert^2_{2,k}+C\Delta t^5 h  \| |\nabla u_{j, ttt}|\|^2_{2,0}
+C\frac{h^{2s+2}}{\nu}\||p_j|\|^{2}_{2,s+1} \nonumber\\
+Ch^{2k+2}\nu^{-1}\| |u_{j,t}|\|^{2}_{2,k+1}
+C\nu h^{2k}\| |\nabla u_{j} | \|_{2,k}^{2}+\frac{C{\Delta t}^{4}}{\nu}\| |
u_{j,ttt}| \| ^{2}_{2,0}\Bigg\rbrace\text{ .}\nonumber
\end{gather}
Applying triangle inequality on the error and absorbing constants gives \eqref{ineq:errlast2}.
\end{proof}

In many applications, e.g., numerical weather prediction, the ensemble mean is usually the main prediction of the future state and thus its behavior is of special interest. Herein we give an error estimate of the ensemble mean computed from our algorithm, which shows the algorithm's ensemble mean does converge to the true ensemble mean with optimal convergence rate.

Let $\langle e \rangle^n =\langle u\rangle^n - \langle u_h\rangle^n$ be the error between the true ensemble mean and the ensemble mean computed from (En-BlendedBDF). Then we have the following error estimate.
\begin{theorem}
[Convergence of ensemble mean]\label{th:errEnMean} Consider the
method (\text{En-BlendedBDF}). If the following conditions hold
\begin{align}
C_e \frac{\Delta t}{\nu h}
\Vert \nabla u_{j,h}^{\prime n}\Vert^2\leq 1,  \qquad j= 1, ..., J,\label{err:CFL}
\end{align}
where $C_e$ is a constant that depends only on the domain and the minimum angle of the mesh and is independent of the timestep, then there is a positive constant $C$
independent of the mesh width and timestep such that
\begin{gather}
\frac{1}{2}\Vert \langle e\rangle^{N}\Vert^{2}
\leq exp\left(\frac{C T}{\nu^2}\right)\frac{1}{J}\sum_{j=1}^J\Bigg \lbrace\frac{1}{2}\Vert e_{j}^{2}\Vert^{2}+\frac{1}{2}\|3e_{j}
^{2}-e_{j}^{1}\|^{2}+\frac{1}{2}\|3e_{j}
^{2}-3e_{j}^{1}+e_{j}^{0}\|^{2}\nonumber\\
+\frac{3\nu\Delta t}{16}\Vert\nabla e_{j}^{2}\Vert^{2}
+\frac{3\nu\Delta t}{8}
\left(\Vert\nabla e_{j}^{2}\Vert^{2}
+\Vert\nabla e_{j}^{1}\Vert^{2}\right)
+\frac{3\nu\Delta t}{16}\Vert\nabla e_{j}^{0}\Vert^{2}\label{en1}\\
+C\frac{h^{2k}}{\nu}\| |\nabla u_j|
\|^2_{\infty, 0}\| |u_j|\|^2_{2, k+1}+C\frac{\Delta t^6}{\nu}\| |\nabla u_{j, ttt}|\|^2_{2,0}
+C \frac{h^{2k}}{\nu}\||\nabla
u_{j}|\|^{2}_{2,k+1} \nonumber\\
+C\Delta t^4 h^{2k+1}\Vert \vert \nabla u_{j, ttt}\vert\Vert^2_{2,k}+C\Delta t^5 h  \| |\nabla u_{j, ttt}|\|^2_{2,0}
+C\frac{h^{2s+2}}{\nu}\||p_j|\|^{2}_{2,s+1} \nonumber\\
+Ch^{2k+2}\nu^{-1}\| |u_{j,t}|\|^{2}_{2,k+1}
+C\nu h^{2k}\| |\nabla u_{j} | \|_{2,k}^{2}+\frac{C{\Delta t}^{4}}{\nu}\| |
u_{j,ttt}| \| ^{2}_{2,0}\Bigg\rbrace\text{ ,}\nonumber
\end{gather}
and 
\begin{gather}
\nu\sum_{n=2}^{N-1}\Delta t
\Vert\nabla \langle e\rangle^{n+1}\Vert^{2}
\leq exp\left(\frac{C T}{\nu^2}\right)\frac{1}{J}\sum_{j=1}^J\Bigg \lbrace\frac{1}{2}\Vert e_{j}^{2}\Vert^{2}+\frac{1}{2}\|3e_{j}
^{2}-e_{j}^{1}\|^{2}+\frac{3\nu\Delta t}{16}\Vert\nabla e_{j}^{2}\Vert^{2}\nonumber\\
+\frac{1}{2}\|3e_{j}
^{2}-3e_{j}^{1}+e_{j}^{0}\|^{2}
+\frac{3\nu\Delta t}{8}
\left(\Vert\nabla e_{j}^{2}\Vert^{2}
+\Vert\nabla e_{j}^{1}\Vert^{2}\right)
+\frac{3\nu\Delta t}{16}\Vert\nabla e_{j}^{0}\Vert^{2}\label{en2}\\
+C\frac{h^{2k}}{\nu}\| |\nabla u_j|
\|^2_{\infty, 0}\| |u_j|\|^2_{2, k+1}+C\frac{\Delta t^6}{\nu}\| |\nabla u_{j, ttt}|\|^2_{2,0}
+C \frac{h^{2k}}{\nu}\||\nabla
u_{j}|\|^{2}_{2,k+1} \nonumber\\
+C\Delta t^4 h^{2k+1}\Vert \vert \nabla u_{j, ttt}\vert\Vert^2_{2,k}+C\Delta t^5 h  \| |\nabla u_{j, ttt}|\|^2_{2,0}
+C\frac{h^{2s+2}}{\nu}\||p_j|\|^{2}_{2,s+1} \nonumber\\
+Ch^{2k+2}\nu^{-1}\| |u_{j,t}|\|^{2}_{2,k+1}
+C\nu h^{2k}\| |\nabla u_{j} | \|_{2,k}^{2}+\frac{C{\Delta t}^{4}}{\nu}\| |
u_{j,ttt}| \| ^{2}_{2,0}\Bigg\rbrace\text{ .}\nonumber
\end{gather}
\end{theorem}
\begin{proof}
With
\begin{align*}
\Vert \langle e \rangle^n\Vert^2=\Vert \frac{1}{J}\sum_{j=1}^J e_j^n\Vert^2\leq \frac{1}{J}\sum_{j=1}^J\Vert e_j^n\Vert^2,
\end{align*}
\eqref{en1} follows directly from Theorem \ref{th:errEnEVBDF2AB2}. Similarly, we have
\begin{align*}
\sum_{n=2}^{N-1}\Vert \langle \nabla e \rangle^n\Vert^2=\sum_{n=2}^{N-1}\Vert \frac{1}{J}\sum_{j=1}^J \nabla e_j^n\Vert^2\leq \sum_{n=2}^{N-1}\frac{1}{J}\sum_{j=1}^J\Vert \nabla e_j^n\Vert^2,
\end{align*}
and thus \eqref{en2} also follows directly from Theorem \ref{th:errEnEVBDF2AB2}.
\end{proof}

}
\section{Numerical Experiments}
We perform numerical experiments for the proposed method on two test problems. First, we verify predicted convergence rates on a 2d test problem with known analytical solution. We also compare accuracy of (En-BlendedBDF) with that of the previously studied (En-BDF2AB2) method (see \cite{J15}). The (En-BDF2AB2) method is given by
\begin{gather}
\frac{3u_{j}^{n+1}-4u_{j}^{n}+u_j^{n-1}}{2\Delta t}+<u>^{n}\cdot\nabla u_{j}^{n+1}\tag{\textsc{En-BDF2AB2}}\label{EnBDF2AB2}\\
\qquad\qquad\qquad\qquad\qquad+u_j^{\prime n}\cdot\nabla \left(2u_{j}^{n}-u_j^{n-1}\right)+\nabla p_{j}^{n+1}-\nu\Delta u_{j}^{n+1}=f_{j}^{n+1}\text{, \ }\nonumber
\\
\nabla\cdot u_{j}^{n+1}=0.\nonumber
\end{gather}
Next, we test the ability of the method to simulate high Reynolds number, complex flows. The method is tested on the well-known 3D Ethier-Steinman flow problem with high Reynolds number and grad-div stabilization is added to the method to relax the time step condition. In all tests, we use Taylor-Hood P2-P1 elements. The code
was implemented using the software package FreeFem++, \cite{H12}.
\subsection{Convergence}\label{convergence}
Our first experiment tests the predicted convergence rates for the method. We take the analytical solution of Navier-Stokes equations from \cite{GQ98}, prescribed in the unit square $\Omega = [0, 1]^2$ 
\begin{align*}
u_{true}& =(-g(t)\cos x\sin y, +g(t)\sin x\cos y)^T, \\
p_{true}& =-\frac{1}{4}[\cos (2x)+\cos (2y)]g^{2}(t),\quad\text{where}\ g(t)=\sin (2t),
\end{align*}%
with source term $f(x,y,t)=[g^{\prime }(t)+2\nu g(t)](-\cos x\sin y, \sin x\cos y)^T$. We take the viscosity $\nu=0.01$ and simulation time $T=1$.  Inhomogeneous Dirichlet boundary condition ${u}=u_{true}\text{ on }\partial\Omega$ is enforced nodally on the boundary.

We consider a set of two realizations of Navier-Stokes equations $u_{1,2}=(1\pm\epsilon)u_{true}$, $\epsilon=10^{-3}$, which correspond to two different initial conditions
$u_{1,2}^{0}=(1\pm\epsilon)u^{0}_{true}$ respectively. In the simulation, the source term and boundary condition for each realization need to be adjusted accordingly. As the method is a three-step method, we need $u^1_1$, $u^1_2$, $u^2_1$, $u^2_2$ as well to get the algorithm start to run. For this test problem, we know the exact solution so we just take the exact solution at each corresponding instant and interpolate it in the finite element space. We then calculate errors and convergence rates by computing approximations with both (En-BlendedBDF) and (En-BDF2) on 5 successive mesh refinements with $h=2\Delta t$. From Table 1 and Table 2, the convergence rate is close to 2, which is optimal according to our theoretical results. In Tables 3, 4, 5 and 6, we compare the error computed with (En-BlendedBDF) and (En-BDF2). As we can see from the tables, the error computed with (En-BlendedBDF) is noticeably smaller as a consequence of smaller temporal errors.

\begin{center}
\end{center}

\begin{table}[t]
\centering
\begin{tabular}
[c]{|l|c|c|c|c|}\hline
$\Delta t$ & $\|u_{1}-u_{1,h}\|_{\infty, 0}$ & rate & $\|\nabla u_{1}- \nabla
u_{1,h}\|_{2,0}$ & rate\\\hline
$0.05$ &  $2.11868\cdot10^{-4}$ & -- & $3.33272\cdot10^{-3}$ &
--\\
$0.025 $ & $5.86519\cdot10^{-5}$ & 1.8529 & $6.46582\cdot
10^{-4}$ & 2.3658\\
$0.0125 $& $1.55198\cdot10^{-5}$ & 1.9181 & $1.50220\cdot
10^{-4}$ & 2.1058\\
$0.00625$ & $3.99025\cdot10^{-6}$ & 1.9596 & $3.72779\cdot
10^{-5}$ & 2.0107\\
$0.003125$ & $1.01142\cdot10^{-6}$ & 1.9800 & $9.36355
\cdot10^{-6}$ & 1.9932\\\hline
\end{tabular}
\caption{(\text{En-BlendedBDF}): Errors and convergence rates for the first ensemble member }%
\label{tab:u1}%
\end{table}

\begin{table}[t]
\centering
\begin{tabular}
[c]{|l|c|c|c|c|}\hline
$\Delta t$ & $\|u_{2}-u_{2,h}\|_{\infty, 0}$ & rate & $\|\nabla u_{2}- \nabla
u_{2,h}\|_{2,0}$ & rate\\\hline
$0.05$ & $2.11487\cdot10^{-4}$ & -- & $3.32141\cdot10^{-3}$
& --\\
$0.025$ & $5.85514\cdot10^{-5}$ & 1.8528& $6.44810\cdot
10^{-4}$ & 2.3648\\
$0.0125 $ & $1.54929\cdot10^{-5}$ & 1.9181& $1.49864\cdot
10^{-4}$ & 2.1052\\
$0.00625 $ & $3.98337\cdot10^{-6}$ & 1.9596 & $3.71937\cdot
10^{-5}$ & 2.0105\\
$0.003125 $ & $1.00968\cdot10^{-6}$ & 1.9801 & $9.34265\cdot
10^{-6}$ & 1.9932\\
\hline
\end{tabular}
\caption{(\text{En-BlendedBDF}): Errors and convergence rates for the second ensemble member}%
\label{tab:u2}%
\end{table}

\begin{table}[t]
\centering
\begin{tabular}
[c]{|l|c|c|c|c|}\hline
$\Delta t$ & En-BlendedBDF  & En-BDF2\\\hline
$0.05$ & $2.11868\cdot10^{-4}$ &  $4.85642\cdot10^{-4}$\\
$0.025$ & $5.86519\cdot10^{-5}$ & $1.26128\cdot
10^{-4}$ \\
$0.0125 $ & $1.55198\cdot10^{-5}$ & $3.21716\cdot
10^{-5}$ \\
$0.00625 $ & $3.99025\cdot10^{-6}$  & $8.12342\cdot
10^{-6}$ \\
$0.003125 $ & $1.01142\cdot10^{-6}$ &  $2.04078\cdot
10^{-6}$ \\
\hline
\end{tabular}
\caption{$\Vert u_1-u_{1,h}\Vert_{\infty, 0}$:  Comparison of (En-BlendedBDF) and (En-BDF2)}
\label{tab:u1Comparison}
\end{table}

\begin{table}[t]
\centering
\begin{tabular}
[c]{|l|c|c|c|c|}\hline
$\Delta t$ & En-BlendedBDF  & En-BDF2\\\hline
$0.05$ & $2.11487\cdot10^{-4}$ &  $4.84794\cdot10^{-4}$\\
$0.025$ & $5.85514\cdot10^{-5}$ & $1.25913\cdot
10^{-4}$ \\
$0.0125 $ & $1.54929\cdot10^{-5}$ & $3.21161\cdot
10^{-5}$ \\
$0.00625 $ & $3.98337\cdot10^{-6}$  & $8.10943\cdot
10^{-6}$ \\
$0.003125 $ & $1.00968\cdot10^{-6}$ &  $2.03726\cdot
10^{-6}$ \\
\hline
\end{tabular}
\caption{$\Vert u_2-u_{2,h}\Vert_{\infty, 0}$:  Comparison of (En-BlendedBDF) and (En-BDF2)}
\label{tab:u2Comparison}
\end{table}

\begin{table}[t]
\centering
\begin{tabular}
[c]{|l|c|c|c|c|}\hline
$\Delta t$ & En-BlendedBDF  & En-BDF2\\\hline
$0.05$ & $3.33272\cdot10^{-3}$ &  $5.11092\cdot10^{-3}$\\
$0.025$ & $6.46582\cdot10^{-4}$ & $1.18810\cdot
10^{-3}$ \\
$0.0125 $ & $1.50220\cdot10^{-4}$ & $2.92502\cdot
10^{-4}$ \\
$0.00625 $ & $3.72779\cdot10^{-5}$  & $7.31031\cdot
10^{-5}$ \\
$0.003125 $ & $9.36355\cdot10^{-6}$ &  $1.83094\cdot
10^{-5}$ \\
\hline
\end{tabular}
\caption{$\Vert \nabla u_1- \nabla u_{1,h}\Vert_{2, 0}$:  Comparison of (En-BlendedBDF) and (En-BDF2)}
\label{tab:u1_H1_Comparison}
\end{table}

\begin{table}[t]
\centering
\begin{tabular}
[c]{|l|c|c|c|c|}\hline
$\Delta t$ & En-BlendedBDF  & En-BDF2\\\hline
$0.05$ & $3.32141\cdot10^{-3}$ &  $5.09708\cdot10^{-3}$\\
$0.025$ & $6.44810\cdot10^{-4}$ & $1.18528\cdot
10^{-3}$ \\
$0.0125 $ & $1.49864\cdot10^{-4}$ & $2.91837\cdot
10^{-4}$ \\
$0.00625 $ & $3.71937\cdot10^{-5}$  & $7.29391\cdot
10^{-5}$ \\
$0.003125 $ & $9.34265\cdot10^{-6}$ &  $1.82684\cdot
10^{-5}$ \\
\hline
\end{tabular}
\caption{$\Vert \nabla u_2-\nabla u_{2,h}\Vert_{2, 0}$:  Comparison of (En-BlendedBDF) and (En-BDF2)}
\label{tab:u2_H1_Comparison}
\end{table}

\subsection{3D Ethier-Steinman Flow}\label{3d}
We test our method on the 3D Ethier-Steinman flow problem for which the analytical solutions are known, \cite{ES94}. The flow has complex structures due to its nontrivial helicity \cite{OR10}, and thus is often used to test numerical methods for Navier-Stokes equations. The 3D analytical solutions on a $[0, 1]^3$ box are given by
\begin{gather}\label{3D:eq}
u_{1} = -a (e^{ax}\sin(ay+dz)+e^{az}\cos(ax+dy))e^{-\nu d^2t}\text{ ,}\\
u_{2} = -a (e^{ay}\sin(az+dx)+e^{ax}\cos(ay+dz))e^{-\nu d^2t}\text{ ,}\nonumber\\
u_{3}= -a (e^{az}\sin(ax+dy)+e^{ay}\cos(az+dx))e^{-\nu d^2t}\text{ ,}\nonumber\\
p =-\frac{a^2}{2}(e^{2ax}+e^{2ay}+e^{2az}+2\sin(ax+dy)\cos(az+dx)e^{a(y+z)}\nonumber\\
+2\sin(ay+dz)\cos(ax+dy)e^{a(z+x)}+2\sin(az+dx)\cos(ay+dz)e^{a(x+y)})e^{-2\nu d^2t}.\nonumber
\end{gather}
Figure \ref{Fig:3d visualization} shows the flow structure of the test problem with streamribbons in the box, velocity streamlines and speed contours on the sides. 
\begin{figure}[htbp]
\begin{center}
\makebox[\textwidth][c]{\includegraphics[height=6.7cm, width=20.0cm]{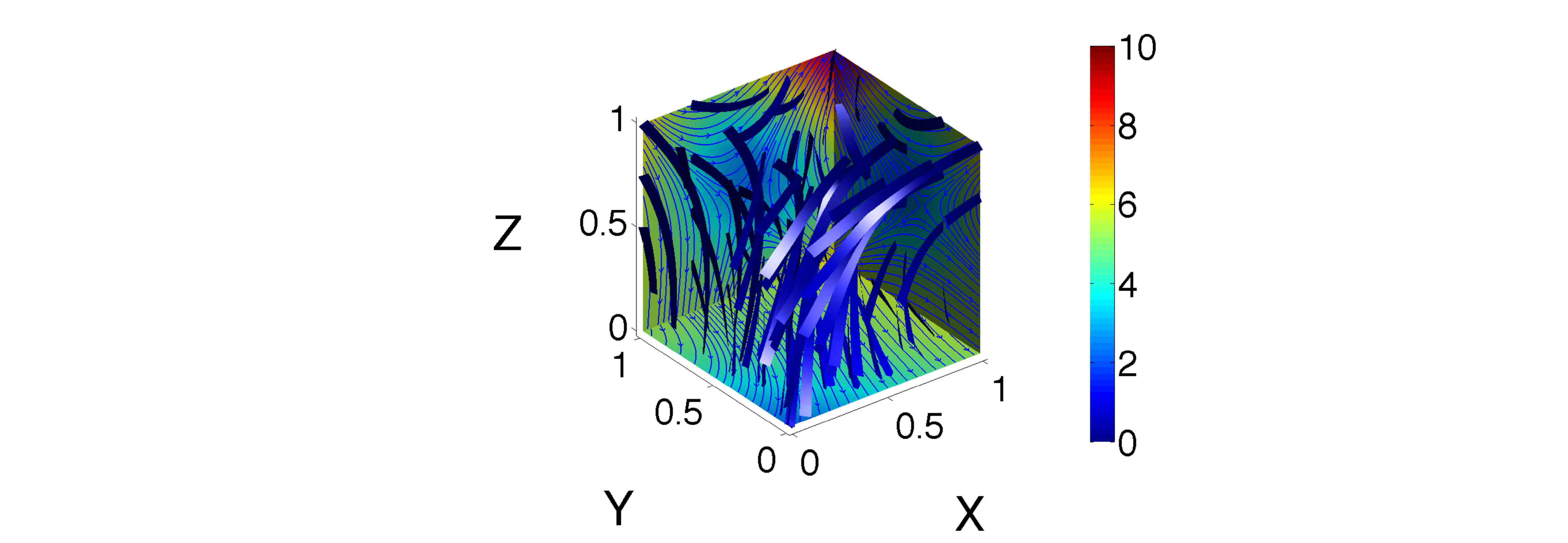}} 
\caption{Flow structure for $\nu=0.001$, $\Delta t=0.02$}
\label{Fig:3d visualization}
\end{center}
\end{figure}

We simulate two realizations this test with perturbed initial conditions generated in the way as in the Section 5.1. The purpose of this test is to show that for high Reynolds number, the time step condition of our method can be relaxed by adding grad-div stabilization $\gamma (\nabla \cdot u^{n+1}_{j,h}, \nabla \cdot v_h) $ and the stabilized method can still give reasonable approximations. As we do not test accuracy here, all tests are run on a relatively coarse mesh and moderately large time steps to save computational time. We take $a=1.25, d=2.25$ and the kinematic viscosity $\nu=0.001$ in \eqref{3D:eq} and consider two realizations with perturbation parameters $\epsilon_{1}=10^{-3}$ and $\epsilon_{2}=-10^{-3}$. The test is run on a coarse mesh with mesh size $h=0.1$. We take time step $\Delta t=0.02$ and run the simulation from $t=0$ to $t=1$. (En-BlendedBDF) encounters numerical instability and the kinetic energy quickly blows up. On the other hand, adding the grad-div stabilization term stabilized the method and gave acceptable approximations. We plot kinetic energy of averaged velocity computed with different stabilization parameter $\gamma$ in Figure \ref{Fig: Energy002}. For $\gamma=0$, which means there is no stabilization, we can see the method is unstable while adding grad-div stabilization makes the method stable and the computed averaged velocity tracks the exact solution pretty well considering the coarse mesh and relatively large time step used. It is worth noting that adding grad-div stabilization introduces numerical errors as one can see from Figure \ref{Fig: Energy002} that the method with $\gamma=0.1$ gives better approximation than the method with $\gamma=1$ which introduces more numerical errors. Nevertheless, if $\gamma$ is too small, it may not be able to stabilize the method, as shown in Figure \ref{Fig: Energy002} the stabilization with $\gamma=0.01$ managed to stabilize the simulation for a short time but the method becomes unstable eventually. The calibration of the stabilization parameter is an essential issue in practice.

\begin{figure}[!h]
\centering
\includegraphics[width=13.5cm, height=4.8cm]{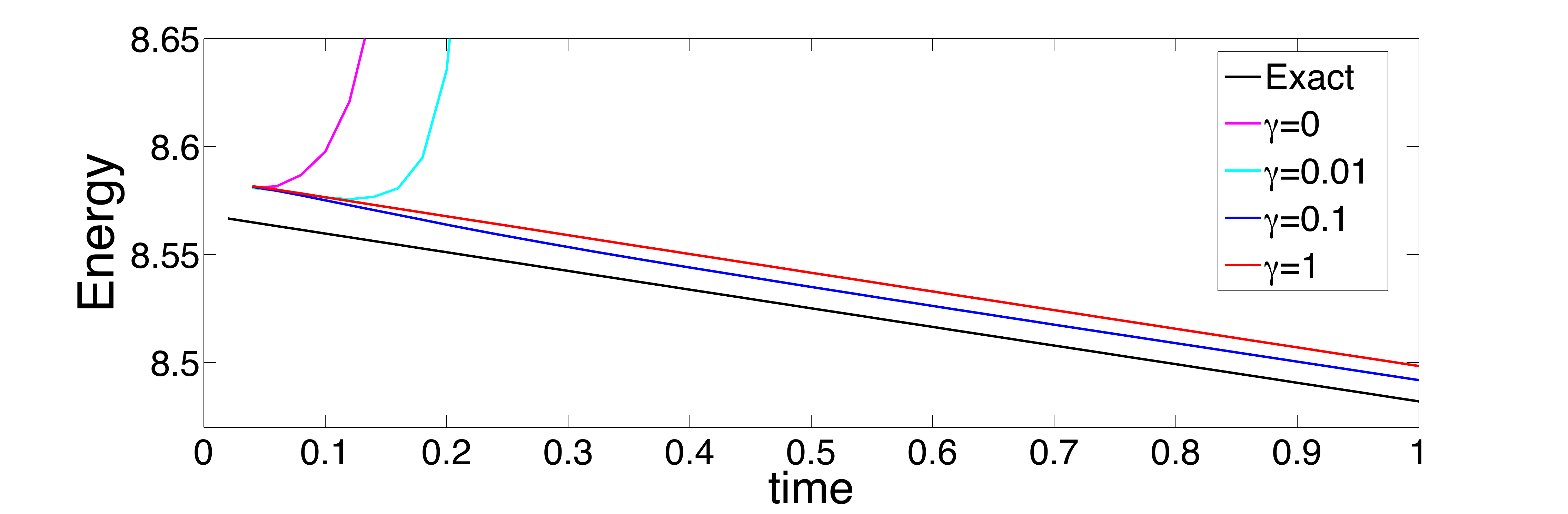} 
\caption{Kinetic Energy for $\nu=0.001$, $\Delta t=0.02$}
\label{Fig: Energy002}
\end{figure}

\section{Conclusion}
The recently developed ensemble simulation methods to efficiently compute an ensemble of fluid flow equations open a new path to quantifying uncertainty and predicting flow behaviors. In this paper, we presented a second order ensemble method based on a blended BDF time stepping scheme with the optimal error constant. This method computes all ensemble members at each timestep in one pass, taking  advantage of the fact that all members have the same coefficient matrix. Compared with the only existing second order method studied in \cite{J15}, this method has noticeably improved accuracy, as is shown in numerical tests. Further research will include applying the method to the computation of the probability distributions of statistics of interest, which are outputs of certain partial differential equations, and investigating regularization methods for flows at high Reynolds number.

\appendix
\section{Proof of Lemma \ref{lm}}\label{App:AppendixA}
{\allowdisplaybreaks
\begin{proof}
To prove \eqref{cons1}, we first rewrite
\begin{gather}
10(u^{n+1}-u^n)-5(u^n-u^{n-1})+(u^{n-1}-u^{n-2})-6\Delta t u_t^{n+1}\nonumber\\
=10\int_{t^n}^{t^{n+1}}u_t dt-5\int_{t^{n-1}}^{t^n}u_t dt +\int_{t^{n-2}}^{t^{n-1}}u_t dt-6\Delta t u_t^{n+1}\nonumber\\
=10\int_{t^n}^{t^{n+1}}\frac{d}{dt}(t-t^n) u_t dt-5\int_{t^{n-1}}^{t^n}\frac{d}{dt}(t-t^{n-1}) u_t dt \nonumber\\
+\int_{t^{n-2}}^{t^{n-1}}\frac{d}{dt}(t-t^{n-2}) u_t dt-6\Delta t u_t^{n+1}\nonumber\\
=10\left[\left[(t-t^n)u_t\right]_{t^n}^{t^{n+1}}-\int_{t^n}^{t^{n+1}}(t-t^n) u_{tt} dt\right]\nonumber\\
-5\left[\left[(t-t^{n-1})u_t\right]_{t^{n-1}}^{t^{n}}-\int_{t^{n-1}}^{t^{n}}(t-t^{n-1}) u_{tt} dt\right]\nonumber\\
+\left[\left[(t-t^{n-2})u_t\right]_{t^{n-2}}^{t^{n-1}}
-\int_{t^{n-2}}^{t^{n-1}}(t-t^{n-2}) u_{tt} dt\right]-6\Delta t u_t^{n+1}\nonumber\\
=\left[4\Delta t u_t^{n+1} - 5\Delta t u_t^n+\Delta t u_t^{n-1}\right]-10\int_{t^n}^{t^{n+1}}\frac{d}{dt}\left(\frac{1}{2}(t-t^n)^2\right) u_{tt} dt\nonumber\\
+5\int_{t^{n-1}}^{t^{n}}\frac{d}{dt}\left(\frac{1}{2}(t-t^{n-1})^2\right) u_{tt} dt-\int_{t^{n-2}}^{t^{n-1}}\frac{d}{dt}\left(\frac{1}{2}(t-t^{n-2})^2\right) u_{tt} dt\nonumber\\
=\left[4\Delta t \int_{t^n}^{t^{n+1}}u_{tt}dt -\Delta t \int_{t^{n-1}}^{t^n}u_{tt}dt\right]\nonumber\\
-10\left[\left[\frac{1}{2}(t-t^n)^2u_{tt}\right]_{t^n}^{t^{n+1}}-\int_{t^n}^{t^{n+1}}\frac{1}{2}(t-t^n)^2u_{ttt}dt\right]\nonumber\\
+5\left[\left[\frac{1}{2}(t-t^{n-1})^2u_{tt}\right]_{t^{n-1}}^{t^{n}}
-\int_{t^{n-1}}^{t^{n}}\frac{1}{2}(t-t^{n-1})^2u_{ttt}dt\right]\nonumber\\
-\left[\left[\frac{1}{2}(t-t^{n-2})^2u_{tt}\right]_{t^{n-2}}^{t^{n-1}}
-\int_{t^{n-2}}^{t^{n-1}}\frac{1}{2}(t-t^{n-2})^2u_{ttt}dt\right]\nonumber\\
=4\Delta t \left[\left[(t-t^n)u_{tt}\right]_{t^n}^{t^{n+1}}-\int_{t^n}^{t^{n+1}}(t-t^n)u_{ttt}dt\right]\nonumber\\
-\Delta t \left[\left[(t-t^{n-1})u_{tt}\right]_{t^{n-1}}^{t^{n}}
-\int_{t^{n-1}}^{t^{n}}(t-t^{n-1})u_{ttt}dt\right]\nonumber\\
-10\left(\frac{1}{2}\Delta t^2 u_{tt}^{n+1}\right)+5\left(\frac{1}{2}\Delta t^2u_{tt}^n\right)-\left(\frac{1}{2}\Delta t^2 u_{tt}^{n-1}\right)\nonumber\\
+10\int_{t^n}^{t^{n+1}}\frac{1}{2}(t-t^n)^2u_{ttt}dt-5\int_{t^{n-1}}^{t^{n}}\frac{1}{2}(t-t^{n-1})^2u_{ttt}dt+\int_{t^{n-2}}^{t^{n-1}}\frac{1}{2}(t-t^{n-2})^2u_{ttt}dt\nonumber\\
=-\frac{1}{2}\Delta t^2 \left[2\int_{t^n}^{t^{n+1}}u_{ttt}dt
-\int_{t^{n-1}}^{t^n}u_{ttt}dt\right]\nonumber\\
-4\Delta t\int_{t^n}^{t^{n+1}}(t-t^n)u_{ttt}dt+\Delta t\int_{t^{n-1}}^{t^{n}}(t-t^{n-1})u_{ttt}dt\nonumber\\
+10\int_{t^n}^{t^{n+1}}\frac{1}{2}(t-t^n)^2u_{ttt}dt-5\int_{t^{n-1}}^{t^{n}}\frac{1}{2}(t-t^{n-1})^2u_{ttt}dt+\int_{t^{n-2}}^{t^{n-1}}\frac{1}{2}(t-t^{n-2})^2u_{ttt}dt\nonumber
\end{gather}
Then the $L^2$ norm of the term of interest can be estimated as follows
\begin{gather}
\Big\Vert\frac{10u^{n+1}-15u^n+6u^{n-1}-u^{n-2}}{6\Delta t}-u_t^{n+1}\Big\Vert^2\\
=\frac{1}{36\Delta t^2}\bigintsss_{\Omega} \Bigg|-\frac{1}{2}\Delta t^2 \left[2\int_{t^n}^{t^{n+1}}u_{ttt}dt
-\int_{t^{n-1}}^{t^n}u_{ttt}dt\right]\nonumber\\
-4\Delta t\int_{t^n}^{t^{n+1}}(t-t^n)u_{ttt}dt+\Delta t\int_{t^{n-1}}^{t^{n}}(t-t^{n-1})u_{ttt}dt
+10\int_{t^n}^{t^{n+1}}\frac{1}{2}(t-t^n)^2u_{ttt}dt\nonumber\\
-5\int_{t^{n-1}}^{t^{n}}\frac{1}{2}(t-t^{n-1})^2u_{ttt}dt+\int_{t^{n-2}}^{t^{n-1}}\frac{1}{2}(t-t^{n-2})^2u_{ttt}dt\Bigg |^2 dx\nonumber\\
\leq\frac{1}{18\Delta t^2}\bigintsss_{\Omega} \Bigg(\Delta t^4 \Bigg|\int_{t^n}^{t^{n+1}}u_{ttt}dt\Bigg |^2
+\frac{1}{4}\Delta t^4\Bigg|\int_{t^{n-1}}^{t^n}u_{ttt}dt\Bigg |^2\nonumber\\
+16\Delta t^2\Bigg|\int_{t^n}^{t^{n+1}}(t-t^n)u_{ttt}dt\Bigg|^2+\Delta t^2\Bigg|\int_{t^{n-1}}^{t^{n}}(t-t^{n-1})u_{ttt}dt\Bigg|^2\nonumber\\
+25\Bigg|\int_{t^n}^{t^{n+1}}(t-t^n)^2u_{ttt}dt\Bigg|^2
+\frac{25}{4}\Bigg|\int_{t^{n-1}}^{t^{n}}(t-t^{n-1})^2u_{ttt}dt\Bigg|^2\nonumber\\
+\frac{1}{4}\Bigg|\int_{t^{n-2}}^{t^{n-1}}(t-t^{n-2})^2u_{ttt}dt\Bigg |^2 \Bigg)dx\nonumber\\
\leq\frac{1}{18\Delta t^2}\bigintsss_{\Omega} \Bigg(\Delta t^5 \int_{t^n}^{t^{n+1}}|u_{ttt}|^2dt
+\frac{1}{4}\Delta t^5\int_{t^{n-1}}^{t^n}|u_{ttt}|^2dt\nonumber\\
+16\Delta t^3\int_{t^n}^{t^{n+1}}|t-t^n|^2|u_{ttt}|^2dt+\Delta t^3\int_{t^{n-1}}^{t^{n}}|t-t^{n-1}|^2|u_{ttt}|^2dt\Bigg|^2\nonumber\\
+25\Delta t \int_{t^n}^{t^{n+1}}(t-t^n)^2|u_{ttt}|^2dt
+\frac{25}{4}\Delta t\int_{t^{n-1}}^{t^{n}}(t-t^{n-1})^2|u_{ttt}|^2dt\nonumber\\
+\frac{1}{4}\Delta t\int_{t^{n-2}}^{t^{n-1}}(t-t^{n-2})^2|u_{ttt}|^2dt \Bigg)dx\nonumber\\
\leq \frac{7}{3}\Delta t^3 \int_{\Omega}\left(\int_{t^{n-2}}^{t^{n+1}}|u_{ttt}|^2 dt\right) dx\leq \frac{7}{3}\Delta t^3\int_{t^{n-2}}^{t^{n+1}}\Vert u_{ttt}\Vert^2 dt.\nonumber
\end{gather}
Now we prove \eqref{cons2}. To start, we rewrite 
\begin{gather}
 \left(u^{n+1}-3u^n+3u^{n-1}-u^{n-2}\right)\\
=\left[ \left(u^{n+1}-u^{n}\right)-\left(u^n-u^{n-1}\right)\right]-\left[ \left(u^{n}-u^{n-1}\right)-\left(u^{n-1}-u^{n-2}\right)\right].\nonumber
\end{gather}
Using integration by parts, the terms in the first brackets in the above equation can written as
\begin{gather}
 \left(u^{n+1}-u^{n}\right)-\left(u^n-u^{n-1}\right)
 =\int_{t^{n}}^{t^{n+1}}u_{t} dt-\int_{t^{n-1}}^{t^n}u_{t}dt\label{consis1}\\
 =\int_{t^{n}}^{t^{n+1}} \frac{d}{dt}(t-t^n)u_{t} dt-\int_{t^{n-1}}^{t^n}\frac{d}{dt}(t-t^{n}) u_{t}dt\nonumber\\
=\left[\left[(t-t^n)u_{t}\right]_{t^n}^{t^{n+1}} -\int_{t^n}^{t^{n+1}}(t-t^n)u_{tt}dt\right]-\left[\left[(t-t^n)u_{t}\right]_{t^{n-1}}^{t^{n}} -\int_{t^{n-1}}^{t^{n}}(t-t^n)u_{tt}dt\right]\nonumber\\
=\left[\Delta t u_{t}^{n+1} -\int_{t^n}^{t^{n+1}}(t-t^n)u_{tt}dt\right]-\left[\Delta t
u_{t}^{n-1} -\int_{t^{n-1}}^{t^{n}}(t-t^n)u_{tt}dt\right]\nonumber\\
=\Delta t \int_{t^{n-1}}^{t^{n+1}}u_{tt}dt -\int_{t^n}^{t^{n+1}}(t-t^n)u_{tt}dt+\int_{t^{n-1}}^{t^{n}}(t-t^n)u_{tt}dt\nonumber\\
=\Delta t \int_{t^{n-1}}^{t^{n+1}}\frac{d}{dt}(t-t^n)u_{tt}dt -\int_{t^n}^{t^{n+1}}\frac{d}{dt}\left(\frac{1}{2} (t-t^n)^2\right)u_{tt}dt+\int_{t^{n-1}}^{t^{n}}\frac{d}{dt}\left(\frac{1}{2}(t-t^n)^2\right) u_{tt}dt\nonumber\\
=\Delta t \left[ \left[(t-t^n)u_{tt}\right]_{t^{n-1}}^{t^{n+1}}-\int_{t^{n-1}}^{t^{n+1}}(t-t^n)u_{ttt}dt\right]\nonumber\\
-\left[\left[\frac{1}{2}(t-t^n)^2u_{tt}\right]_{t^n}^{t^{n+1}}-\int_{t^{n}}^{t^{n+1}}\left(\frac{1}{2}(t-t^n)^2\right) u_{ttt}dt \right]\nonumber\\
+\left[\left[\frac{1}{2}(t-t^n)^2u_{tt}\right]_{t^{n-1}}^{t^{n}}-\int_{t^{n-1}}^{t^{n}}\left(\frac{1}{2}(t-t^n)^2\right) u_{ttt}dt \right]\nonumber\\
=\Delta t \left[ \Delta t \left(u_{tt}^{n+1}+u_{tt}^{n-1}\right)-\int_{t^{n-1}}^{t^{n+1}}(t-t^n)u_{ttt}dt\right]\nonumber\\
-\left[\left(\frac{1}{2} \Delta t^2\right)u_{tt}^{n+1}-\int_{t^{n}}^{t^{n+1}}\left(\frac{1}{2}(t-t^n)^2\right) u_{ttt}dt \right]\nonumber\\
+\left[\left(-\frac{1}{2}\Delta t^2\right)u_{tt}^{n-1}-\int_{t^{n-1}}^{t^{n}}\left(\frac{1}{2}(t-t^n)^2\right) u_{ttt}dt \right].\nonumber
\end{gather}
Similarly, we have
\begin{gather}
 \left(u^{n}-u^{n-1}\right)-\left(u^{n-1}-u^{n-2}\right)\label{consis2}\\
=\Delta t \left[ \Delta t \left(u_{tt}^{n}+u_{tt}^{n-2}\right)-\int_{t^{n-2}}^{t^{n}}(t-t^{n-1})u_{ttt}dt\right]\nonumber\\
-\left[\left(\frac{1}{2} \Delta t^2\right)u_{tt}^{n}-\int_{t^{n-1}}^{t^{n}}\left(\frac{1}{2}(t-t^{n-1})^2\right) u_{ttt}dt \right]\nonumber\\
+\left[\left(-\frac{1}{2}\Delta t^2\right)u_{tt}^{n-2}-\int_{t^{n-2}}^{t^{n-1}}\left(\frac{1}{2}(t-t^{n-1})^2\right) u_{ttt}dt \right].\nonumber
\end{gather}
Subtracting \eqref{consis2} from \eqref{consis1} gives
\begin{gather}
 u^{n+1}-3u^n+3u^{n-1}-u^{n-2}\\
=\Delta t \left[ \Delta t \left(u_{tt}^{n+1}-u_{tt}^n+u_{tt}^{n-1}-u_{tt}^{n-2}\right)-\int_{t^{n-1}}^{t^{n+1}}(t-t^n)u_{ttt}dt+\int_{t^{n-2}}^{t^{n}}(t-t^{n-1})u_{ttt}dt\right]\nonumber\\
-\left[\left(\frac{1}{2} \Delta t^2\right)u_{tt}^{n+1}-\left(\frac{1}{2} \Delta t^2\right)u_{tt}^{n}+\left(\frac{1}{2}\Delta t^2\right)u_{tt}^{n-1}-\left(\frac{1}{2} \Delta t^2\right)u_{tt}^{n-2}\right]\nonumber\\
+\left[\int_{t^{n}}^{t^{n+1}}\left(\frac{1}{2}(t-t^n)^2\right) u_{ttt}dt -\int_{t^{n-1}}^{t^{n}}\left(\frac{1}{2}(t-t^n)^2\right) u_{ttt}dt \right]\nonumber\\
-\left[\int_{t^{n-1}}^{t^{n}}\left(\frac{1}{2}(t-t^{n-1})^2\right) u_{ttt}dt -\int_{t^{n-2}}^{t^{n-1}}\left(\frac{1}{2}(t-t^{n-1})^2\right) u_{ttt}dt \right]\nonumber\\
=\Delta t \left[ \frac{1}{2}\Delta t \left(u_{tt}^{n+1}-u_{tt}^n+u_{tt}^{n-1}-u_{tt}^{n-2}\right)-\int_{t^{n-1}}^{t^{n+1}}(t-t^n)u_{ttt}dt+\int_{t^{n-2}}^{t^{n}}(t-t^{n-1})u_{ttt}dt\right]\nonumber\\
+\left[\int_{t^{n}}^{t^{n+1}}\left(\frac{1}{2}(t-t^n)^2\right) u_{ttt}dt -\int_{t^{n-1}}^{t^{n}}\left(\frac{1}{2}(t-t^n)^2\right) u_{ttt}dt \right]\nonumber\\
-\left[\int_{t^{n-1}}^{t^{n}}\left(\frac{1}{2}(t-t^{n-1})^2\right) u_{ttt}dt -\int_{t^{n-2}}^{t^{n-1}}\left(\frac{1}{2}(t-t^{n-1})^2\right) u_{ttt}dt \right]\nonumber\\
=\Delta t \left[ \frac{1}{2}\Delta t\left(\int_{t^{n}}^{t^{n+1}}u_{ttt} dt+\int_{t^{n-2}}^{t^{n-1}}u_{ttt}dt\right)-\int_{t^{n-1}}^{t^{n+1}}(t-t^n)u_{ttt}dt+\int_{t^{n-2}}^{t^{n}}(t-t^{n-1})u_{ttt}dt\right]\nonumber\\
+\left[\int_{t^{n}}^{t^{n+1}}\left(\frac{1}{2}(t-t^n)^2\right) u_{ttt}dt -\int_{t^{n-1}}^{t^{n}}\left(\frac{1}{2}(t-t^n)^2\right) u_{ttt}dt \right]\nonumber\\
-\left[\int_{t^{n-1}}^{t^{n}}\left(\frac{1}{2}(t-t^{n-1})^2\right) u_{ttt}dt -\int_{t^{n-2}}^{t^{n-1}}\left(\frac{1}{2}(t-t^{n-1})^2\right) u_{ttt}dt \right].\nonumber
\end{gather}
Then by the Cauchy-Schwarz inequality we have
\begin{gather}
\Vert \nabla \left(u^{n+1}-3u^n+3u^{n-1}-u^{n-2}\right)\Vert^2\\
=\bigintsss_{\Omega} \Bigg |   \frac{1}{2}\Delta t^2\left(\int_{t^{n}}^{t^{n+1}}\nabla u_{ttt} dt
+\int_{t^{n-2}}^{t^{n-1}}\nabla u_{ttt}dt\right)\nonumber\\
-\Delta t\left[\int_{t^{n-1}}^{t^{n+1}}(t-t^n)\nabla u_{ttt}dt-\int_{t^{n-2}}^{t^{n}}(t-t^{n-1})\nabla u_{ttt}dt\right]\nonumber\\
+\left[\int_{t^{n}}^{t^{n+1}}\left(\frac{1}{2}(t-t^n)^2\right) \nabla u_{ttt}dt -\int_{t^{n-1}}^{t^{n}}\left(\frac{1}{2}(t-t^n)^2\right) \nabla u_{ttt}dt \right]\nonumber\\
-\left[\int_{t^{n-1}}^{t^{n}}\left(\frac{1}{2}(t-t^{n-1})^2\right) \nabla u_{ttt}dt -\int_{t^{n-2}}^{t^{n-1}}\left(\frac{1}{2}(t-t^{n-1})^2\right) \nabla u_{ttt}dt \right]\Bigg |^2 dx\nonumber\\
\leq 2\bigintsss_{\Omega}    \frac{1}{4}\Delta t^4 \Big |\int_{t^{n}}^{t^{n+1}}\nabla u_{ttt} dt \Big |^2
+\frac{1}{4}\Delta t^4\Big |\int_{t^{n-2}}^{t^{n-1}}\nabla u_{ttt}dt\Big |^2\nonumber\\
+\Delta t^2\Big |\int_{t^{n-1}}^{t^{n+1}}(t-t^n)\nabla u_{ttt}dt\Big |^2+\Delta t^2\Big |\int_{t^{n-2}}^{t^{n}}(t-t^{n-1})\nabla u_{ttt}dt\Big |^2\nonumber\\
+\Big |\int_{t^{n}}^{t^{n+1}}\left(\frac{1}{2}(t-t^n)^2\right) \nabla u_{ttt}dt \Big |^2+\Big |\int_{t^{n-1}}^{t^{n}}\left(\frac{1}{2}(t-t^n)^2\right) \nabla u_{ttt}dt \Big |^2\nonumber\\
+\Big |\int_{t^{n-1}}^{t^{n}}\left(\frac{1}{2}(t-t^{n-1})^2\right) \nabla u_{ttt}dt\Big |^2 +\Big |\int_{t^{n-2}}^{t^{n-1}}\left(\frac{1}{2}(t-t^{n-1})^2\right) \nabla u_{ttt}dt \Big |^2 dx\nonumber\\
\leq 2\bigintsss_{\Omega}  \Bigg\{  \frac{1}{4}\Delta t^5\int_{t^{n}}^{t^{n+1}} |\nabla u_{ttt}|^2 dt 
+\frac{1}{4}\Delta t^5\int_{t^{n-2}}^{t^{n-1}} |\nabla u_{ttt}|^2dt\nonumber\\
+2\Delta t^5\int_{t^{n-1}}^{t^{n+1}}|\nabla u_{ttt}|^2dt+2\Delta t^5\int_{t^{n-2}}^{t^{n}}|\nabla u_{ttt}|^2dt\nonumber\\
+\frac{1}{4}\Delta t^5\int_{t^{n}}^{t^{n+1}}|\nabla  u_{ttt}|^2 dt+\frac{1}{4}\Delta t^5\int_{t^{n-1}}^{t^{n}}|\nabla u_{ttt}|^2dt \nonumber\\
+\frac{1}{4}\Delta t^5\int_{t^{n-1}}^{t^{n}}| \nabla u_{ttt}|^2dt  +\frac{1}{4}\Delta t^5\int_{t^{n-2}}^{t^{n-1}}| \nabla u_{ttt}|^2 dt  \Bigg \} dx\nonumber\\
\leq 2\int_{\Omega}   \frac{9}{2}\Delta t^5\int_{t^{n-2}}^{t^{n+1}} |\nabla u_{ttt}|^2 dt  dx\nonumber
\leq 9\Delta t^5 \int_{t^{n-2}}^{t^{n+1}}  \Vert \nabla u_{ttt}\Vert^2 dt.
\end{gather}
This completes the proof.
\end{proof}
}

\end{document}